\documentclass{amsart}

\usepackage{amssymb}
\usepackage[all]{xy}
\usepackage{epsfig} %% replace with graphx
\usepackage{tabularx}
\usepackage{enumerate}

\newcolumntype{C}{>{\centering\arraybackslash}X}

\def\comment#1{{\sf{[#1]}}}

\def\Z{{\mathbb Z}}
\def\Q{{\mathbb Q}}
\def\R{{\mathbb R}}
\def\C{{\mathbb C}}
\def\N{{\mathbb N}}
\def\P{{\mathbb P}}

\def\H{{\mathbb H}}

\def\L{{\mathbb L}}

\def\V{{\mathbb V}}

\def\A{{\mathcal A}}
\def\cC{{\mathcal C}}

\def\cG{{\mathcal G}}
\def\cH{{\mathcal H}}
\def\I{{\mathcal I}}

\def\M{{\mathcal M}}
\def\cO{{\mathcal O}}
\def\cP{{\mathcal P}}
\def\T{{\mathcal T}}
\def\U{{\mathcal U}}

\def\X{{\mathcal X}}

\def\D{\Delta}
\def\G{\Gamma}

\def\a{\alpha}
\def\b{\beta}

\def\d{\delta}

\def\f{{\mathfrak f}}
\def\g{{\mathfrak g}}
\def\h{{\mathfrak h}}

\def\n{{\mathfrak n}}
\def\p{{\mathfrak p}}
\def\r{{\mathfrak r}}

\def\t{{\mathfrak t}}
\def\u{{\mathfrak u}}

\def\that{\hat{\t}}
\def\uhat{\hat{\u}}
\def\rhohat{{\hat{\rho}}}
\def\phihat{{\hat{\phi}}}

\def\phitilde{\tilde{\varphi}}

\def\hbar{\overline{h}}

\def\Ghat{\widehat{\G}}

\def\Khat{\widehat{K}}
\def\That{\widehat{T}}

\def\sp{\mathfrak{sp}}

\def\Ga{{\mathbb{G}_a}}
\def\Sp{{\mathrm{Sp}}}
\def\SL{{\mathrm{SL}}}
\def\GL{{\mathrm{GL}}}

\def\Rep{{\mathrm{Rep}}}

\def\LCS{\mathrm{LCS}}
\def\dot{\bullet}
\def\bs{\backslash}

\def\blank{\phantom{x}}

\def\bil{{\langle \blank,\blank\rangle}}
\def\bracket{{[ \blank,\blank]}}

\def\LiE{{\sf{LiE}}}

\def\Rdual{\check{R}}

\newcommand\im{\operatorname{im}}
\newcommand\Span{\operatorname{span}}

\newcommand\Ad{\operatorname{Ad}}
\newcommand\ad{\operatorname{ad}}
\newcommand\Hom{\operatorname{Hom}}
\newcommand\End{\operatorname{End}}
\newcommand\Diff{\operatorname{Diff}}
\newcommand\Aut{\operatorname{Aut}}
\newcommand\Out{\operatorname{Out}}
\newcommand\Der{\operatorname{Der}}
\newcommand\OutDer{\operatorname{OutDer}}
\newcommand\Gr{\operatorname{Gr}}

\newcommand\coker{\operatorname{coker}}

\newtheorem{theorem}{Theorem}[section]
\newtheorem{lemma}[theorem]{Lemma}
\newtheorem{proposition}[theorem]{Proposition}
\newtheorem{corollary}[theorem]{Corollary}
\newtheorem{bigtheorem}{Theorem}

\theoremstyle{definition}

\newtheorem{conjecture}[theorem]{Conjecture}

\theoremstyle{remark}

\newtheorem{question}[theorem]{Question}

\begin{document}

\title{Genus $3$ Mapping Class Groups are not K\"ahler}

\author{Richard Hain}
\address{Department of Mathematics\\ Duke University\\
Durham, NC 27708-0320}
\email{hain@math.duke.edu}

\thanks{Supported in part by grant DMS-1005675 from the National Science
Foundation}

\date{\today}


\begin{abstract}
We prove that finite index subgroups of genus $3$ mapping class and Torelli
groups that contain the group generated by Dehn twists on bounding simple closed
curves are not K\"ahler. These results are deduced from explicit presentations
of the unipotent (aka, Malcev) completion of genus 3 Torelli groups and of the
relative completions of genus 3 mapping class groups. The main results follow
from the fact that these presentations are not quadratic. To complete the
picture, we compute presentations of completed Torelli and mapping class in
genera $\ge 4$; they are quadratic. We also show that groups commensurable with
hyperelliptic mapping class groups and mapping class groups in genera $\le 2$
are not K\"ahler.
\end{abstract}

\maketitle

\tableofcontents

\section{Introduction}

Mapping class groups are groups of topological symmetries of a compact surface.
More precisely, suppose that $g$ and $n$ are non-negative integers satisfying
$2g-2+n>0$. Let $S$ be a compact oriented surface of genus $g$ and $P$ a set of
$n$ points on $S$. The mapping class group $\G_{g,n}$ is defined to be the group
$$
\G_{g,n} := \pi_0\Diff^+(S,P),
$$
of isotopy classes of orientation preserving diffeomorphisms of $S$ that fix $P$
pointwise. The Torelli group $T_{g,n}$ is the set of elements of $\G_{g,n}$ that
act trivially on $H_1(S;\Z)$. Johnson \cite{johnson:fg} proved that $T_{g,n}$ is
finitely generated for all $g\ge 3$. The {\em Johnson group} $K_{g,n}$  is the
subgroup of $T_{g,n}$ generated by Dehn twists on bounding simple closed curves
(BSCC maps, for short).

A {\em K\"ahler group} is a group that can be realized as the fundamental group
of a compact K\"ahler manifold. Groups commensurable with mapping class groups
in genus $\le 2$ are not K\"ahler.\footnote{The genus $2$ case was established
by Veliche \cite{veliche}. The cases in genera $\le 1$ seem to be folklore. We
give complete proofs in the appendix and also prove that groups commensurable
with hyperelliptic mapping class groups of any genus cannot be K\"ahler.} Our
first result is that genus $3$ mapping class groups are not K\"ahler. Suppose
$n\ge 0$.

\begin{bigtheorem}
\label{thm:main_mcg}
No finite index subgroup of $\G_{3,n}$ that contains $K_{3,n}$ can be the
fundamental group of a compact K\"ahler manifold.
\end{bigtheorem}

Finite index subgroups of $\G_{g,n}$ that contain $K_{g,n}$ when $g\ge 3$
include fundamental groups of moduli spaces of curves with an abelian (i.e.,
standard) level structure; moduli spaces of curves with a Prym-level structure
as defined by Looijenga \cite{looijenga} and Boggi-Pikaart \cite{boggi-pikaart};
and fundamental groups of moduli spaces of curves with an arbitrary root of
their canonical bundle. 

An easier result is that genus 3 Torelli groups are not K\"ahler. As above, $n$
is a non-negative integer.

\begin{bigtheorem}
\label{thm:main_torelli}
No finite index subgroup of $T_{3,n}$ that contains $K_{3,n}$ can the be
fundamental group of a compact K\"ahler manifold.
\end{bigtheorem}

To prove the second result, we prove that the unipotent (aka Malcev) completion
of each genus 3 Torelli group is not quadratically presented.\footnote{This
result was stated informally in \cite[Rem.~10.4]{hain:torelli}.} One then
concludes that these groups are not K\"ahler by applying the well known result
of Deligne, Griffiths, Morgan and Sullivan \cite{dgms}. This computation is
extended to finite index subgroups of genus 3 Torelli groups that contain the
Johnson subgroups using recent work of Andrew Putman \cite{putman}.

The proof of Theorem~\ref{thm:main_mcg} is more involved. The first step is to
compute an explicit presentation of the unipotent radical $\u_3$ of the relative
completion of the genus $3$ mapping class group $\G_3$. This we deduce from an
explicit presentation of the unipotent completion of the genus $3$ Torelli group
in Section~\ref{sec:pres_3} using computations from Section~\ref{sec:commuting}
and the work of Sakasai \cite{sakasai}. Dennis Johnson's work \cite{johnson:h1}
implies that the standard representation $\rho : \G_3 \to \Sp_3(\C)$ is rigid. 
Results of Carlos Simpson \cite{simpson} then imply that if $\G_3$ is the
fundamental group of a compact K\"ahler manifold $X$, then $\rho$ is the
monodromy representation of a polarized variation of Hodge structure over $X$. A
generalization of the Deligne, Griffiths, Morgan, Sullivan result to relative
completion, proved in \cite{hain:malcev}, implies that if $\G_3$ is K\"ahler,
then $\u_3$ would be quadratically presented. Since it is not, $\G_3$ is not
K\"ahler.

This paper may be regarded as a sequel to \cite{hain:torelli} where explicit
presentations of the unipotent completion of Torelli groups and the unipotent
radical of relative completions of mapping class groups were given for all $g\ge
6$ and partial presentations were derived when $3\le g < 6$. Here we complete
the story by computing explicit presentations when $3 \le g < 6$.

Before stating the quadratic presentations we need to fix notation. Denote the
first rational homology $H_1(S,\Q)$ of the reference surface $S$ by $H$. The
symplectic group $\Sp_g$ of rank $g$, also denoted $\Sp(H)$, is the group of
automorphisms of $H$ that preserve the intersection pairing. The third
fundamental representation of $\Sp(H)$ is $\Lambda^3_0 H := (\Lambda^3 H)/H$. It
plays a distinguished role as the Johnson homomorphism induces a canonical
$\Sp_g(\Z)$-equivariant isomorphism $H_1(T_g;\Q) \cong \Lambda^3_0 H$.

The free Lie algebra generated by the vector space $V$ will be denoted $\L(V)$.
It is graded; $\L_n(V)$ will denote its space of degree $n$ elements. Denote the
lower central series of a Lie algebra $\g$ by
$$
\g = L^1 \g \supseteq L^2 \g \supseteq L^3 \g \supseteq \cdots
$$
and the associated graded Lie algebra by $\Gr^\dot_\LCS\g$.

The Lie algebra of the unipotent completion of $T_g$ will be denoted by $\t_g$
and the Lie algebra of the prounipotent radical of the completion of the mapping
class group $\G_g$ with respect to the homomorphism $\G_g \to \Sp_g$ will be
denoted by $\u_g$.

The following Theorem is proved in Section~\ref{sec:quad_relns}. The main
ingredients in the proof are Hodge theory, which implies that each of these Lie
algebras is isomorphic to the degree completion of the associated graded Lie
algebra of its lower central series; Kabanov purity \cite{kabanov:purity} (see
also Section~\ref{sec:kabanov}), which implies that the degrees of relations in
a minimal presentation of these Lie algebras is $\le 3$; and computations of
Sakasai \cite{sakasai}, which extend computations from \cite{hain:torelli}. The
result when $g\ge 6$ was previously proved in \cite{hain:torelli}.

\begin{bigtheorem}
\label{thm:torelli}
If $g\ge 4$, then
$$
\t_g \cong \prod_{n\ge 1} \Gr^n_\LCS \t_g.
$$
and $\Gr^n_\LCS \t_g$ has the quadratic presentation
$$
\Gr^\dot_\LCS \t_g \cong \L(\Lambda^3_0 H)/(R)
$$
in the category of graded Lie algebras in the category of $\Sp_g$-modules, where
$R$ is the $\Sp(H)$ complement of the unique copy of $V(2\lambda_2)$ in
$\L_2(\Lambda^3_0 H)$. Similarly, $\u_g$ is isomorphic to the degree completion
of the graded Lie algebra associated to its lower central series. It is has
presentation
$$
\Gr^\dot_\LCS \u_g \cong \L(\Lambda^3_0 H)/(R+V(0))
$$
in the category of $\Sp_g$-modules, where $V(0)$ is the unique copy of the
trivial representation in $\L_2(\Lambda^3_0 H)$.
\end{bigtheorem}

From this one can deduce presentations of all $\t_{g,n}$ and $\u_{g,n}$ --- the
Lie algebras of the unipotent completion of $T_{g,n}$ and of the prounipotent
radical of the relative completion of $\G_{g,n}$ --- for all $n > 0$ when $g\ge
4$. These are also  quadratically presented. This result allows Theorem B of
\cite{dhp} to be extended from $g\ge 6$ to all $g\ge 4$.

The derivation of the genus 3 presentations is more difficult and involves the
computation of the cubic relations generated by a commuting pair consisting of a
Dehn twist on a bounding simple closed curve and a bounding pair map. These
computations are carried out in Section~\ref{sec:commuting} and also use
Sakasai's computations \cite{sakasai}. As in higher genus case, the proof uses
Hodge theory and Kabanov purity.

\begin{bigtheorem}
\label{thm:torelli3}
In genus $3$, the Lie algebras $\t_3$ and $\u_3$ are isomorphic to the degree
completions of the graded Lie algebras associated to their lower central series:
$$
\t_3 \cong \prod_{n\ge 1} \Gr^n_\LCS \t_3 \text{ and }
\u_3 \cong \prod_{n\ge 1} \Gr^n_\LCS \u_3
$$
The graded Lie algebra $\Gr^n_\LCS \t_3$ has the cubic presentation
$$
\Gr^n_\LCS \t_3 \cong \L(\Lambda^3_0 H)/(R)
$$
in the category of graded Lie algebras in the category of $\Sp_3$-modules, where
$R$ is the $\Sp(H)$ complement of the unique copy of $V(2\lambda_1+\lambda_3)$
in $\L_3(\Lambda^3_0 H)$. The Lie algebra $\Gr^\dot_\LCS\u_3$ has presentation
$$
\Gr^n_\LCS \u_3 \cong \L(\Lambda^3_0 H)/(R+V(0))
$$
in the category of $\Sp_3$-modules, where $V(0)$ is the unique copy of the
trivial representation in $\L_2(\Lambda^3_0 H)$.
\end{bigtheorem}

One consequence of the computations is that the cup product $\Lambda^2
H^1(T_3;\Q) \to H^2(T_3;\Q)$ vanishes, which implies that there is a
well-defined  Massey triple product map
$$
H^1(T_3;\Q)^{\otimes 3} \to H^2(T_3;\Q).
$$
The computation of the cubic relations in $\t_3$ is dual to the computation
of the Massey triple products.

\begin{bigtheorem}
The image of the Massey triple product map is an $\Sp_3(\Z)$ submodule
isomorphic to the restriction of the $\Sp_3$-module $V(2\lambda_2+\lambda_3) +
V(\lambda_1+\lambda_2) + V(\lambda_3)$ to $\Sp_3(\Z)$.
\end{bigtheorem}

A word about exposition: I could have written a shorter paper. However, in the
interests of making the proofs of the main results comprehensible and the paper
reasonably coherent, I have included expository material, especially where I
felt I could improve on earlier expositions. In the final section I speculate on
how the problem of deciding whether higher genus mapping class groups are
K\"ahler might be related to the problem of understanding the topology of
complete subvarieties of moduli spaces of curves. In particular, I conjecture
that the fundamental group of a smooth complete subvariety of $\M_g$ cannot be a
finite index subgroup of $\G_g$.

\bigskip

\noindent{\em Acknowledgments:} I am grateful to Carlos Simpson for helpful
discussions related to his work. I am also indebted to the referee for his/her
careful reading of the manuscript and for pointing out numerous typos in the
original manuscript.

\section{Preliminaries}

This section is included for clarity. In it we introduce some notation and
conventions, and recall several definitions.  Readers should skip this section
and refer to it as needed.

\subsection{Filtrations}

Increasing filtrations
$$
0 \subseteq \cdots \subseteq F_m V \subseteq F_{m+1} V \subseteq
\cdots \subseteq V
$$
of $V$ (a vector space, group, etc) will be denoted with a lower index,
$F_\dot V$. Decreasing filtrations
$$
U \supseteq \cdots \supseteq F^m U \supseteq F^{m+1}U \supseteq
\cdots\supseteq 0
$$
of $U$ (a vector space, a group, etc) will be denoted with an upper index,
$F^\dot U$. The notation for the graded quotients of $F_\dot V$ and $F^\dot U$
is
$$
\Gr_m^F V := F_m V/F_{m-1}V\text{ and } \Gr^m_F U := F^m U/F^{m+1}U.
$$

The lower central series of a Lie algebra (or group) $\g$ will be denoted by
$L^\dot \g$.\footnote{When $\g$ is a topological Lie algebra (or group), $L^m\g$
is defined inductively to be the closure of $[\g,L^{m-1}\g]$.} It is indexed so
that $\g = L^1\g$. With this convention the bracket preserves $L^\dot$:
$$
[L^m\g,L^n\g] \subseteq L^{m+n}\g.
$$

\subsection{Proalgebraic groups}

Suppose that $F$ is a field of characteristic zero. An affine proalgebraic group
$G$ over $F$ is an inverse limit of affine algebraic $F$-groups $G_\a$. The
coordinate ring $\cO(G)$ of $G$ the direct limit of the coordinate rings of the
$G_\a$. The Lie algebra $\g$ of $G$ is the inverse limit of the Lie algebras
$\g_\a$ of the $G_\a$.  It is a Hausdorff topological Lie algebra. The
neighbourhoods of $0$ are the kernels of the canonical projections $\g \to
\g_\a$.

The continuous cohomology of $\g=\varprojlim \g_\a$ is defined by
$$
H^\dot(\g) := \varinjlim_\a H^\dot(\g_\a).
$$
Its homology is the full dual:
$$
H_\dot(\g) := \Hom_F(H^\dot(\g),F) \cong \varprojlim H_\dot(\g_\a)
$$
Each homology group is a Hausdorff topological vector space.

Continuous cohomology can be computed using continuous Chevalley-Eilenberg
cochains:
$$
C^\dot(\g) := \varinjlim_\a \Hom_F(\Lambda^\dot\g_\a,F)
$$
with the usual differential.

If, instead, $\g=\bigoplus_m \g_m$ is a graded Lie algebra, then the homology
and cohomology of $\g$ are also graded. This follows from the fact that the
grading of $\g$ induces a grading of the Chevalley-Eilenberg chains and cochains
of $\g$. 

\subsection{Prounipotent groups and pronilpotent Lie algebras}

A prounipotent $F$-group is a proalgebraic group that is an inverse limit of
unipotent $F$-groups.

A pronilpotent Lie algebra over a $F$ is an inverse limit of finite dimensional
nilpotent Lie algebras. The Lie algebra of a prounipotent groups is a
pronilpotent Lie algebra. The functor that takes a prounipotent group to its Lie
algebra is an equivalence of categories between the category of unipotent
$F$-groups and the category of pronilpotent Lie algebras over $F$. 

\subsection{Hodge theory}

The reader is assumed to be familiar with the basic properties of mixed Hodge
structures (MHSs), which can be found in \cite{deligne:hodge2}. One property
that we will exploit is that there is a natural (though not canonical)
isomorphism
$$
V_\C \cong \bigoplus_{m\in \Z} \Gr^W_\dot V_\C
$$
of the complex part $V_\C$ of a MHS $V$ with the direct sum its weight graded
quotients. These isomorphisms are compatible with tensor products and duals.

A Lie algebra in the category of (rational or real) MHS is a finite dimensional
Lie algebra endowed with a MHS with the property that the bracket is a morphism
of MHS. The cohomology of an inverse limit $\varprojlim \g_\a$ of Lie algebras
in the category of MHS is an Ind-MHS --- that is, an ind-object of the category
of MHS.

If the weight graded quotients of $\g$ are finite dimensional, then the
exactness of the functor $\Gr^W_\dot$ on the category of MHS implies that there
are natural isomorphisms
$$
\Gr^W_\dot H_\dot(\g) \cong H_\dot(\Gr^W_\dot \g)
\text{ and }
\Gr^W_\dot H^\dot(\g) \cong H^\dot(\Gr^W_\dot\g).
$$

\section{Presentations of Pronilpotent Lie Algebras}

Here we recall how each pronilpotent and each positively graded Lie algebra $\n$
has a ``minimal presentation'' with generating set isomorphic to $H_1(\n)$ and
relations isomorphic to $H_2(\n)$. Denote the free Lie algebra (over $F$)
generated by a vector space $V$ by $\L(V)$. This is a graded Lie algebra: degree
$m$ component $\L_m(V)$ consists of the homogeneous Lie words of degree $m$. The
$m$th term of its lower central series is
$$
L^m\L(V) = \bigoplus_{r\ge m} \L_r(V).
$$
The lower central series filtration defines a topology on $\L(V)$. Its
completion in this topology is the degree completion of $\L(V)$:
$$
\L(V)^\wedge = \varprojlim_m \L(V)/L^m \cong \prod_{m\ge 1} \L_m(V).
$$ 
When $V$ is finite dimensional, $\L(V)^\wedge$ is pronilpotent.

Suppose that $\n$ is a pronilpotent Lie algebra. For simplicity, we assume that
$H_1(\n)$ is finite dimensional. Each choice of a continuous splitting of the
canonical surjection $\n \to H_1(\n)$ induces a continuous surjection
$$
\L(H_1(\n))^\wedge \to \n.
$$
The ideal of relations $\r$ is the kernel of this homomorphism. Note that $\r$
is contained in $L^2\L(V)^\wedge$. By the Lie algebra analogue of a theorem
of Hopf (cf.\ \cite[Prop.~5.6]{hain:torelli}) there is a natural isomorphism
\begin{equation}
\label{eqn:hopf}
H_2(\n) \cong \r/[\r,\L(V)^\wedge].
\end{equation}
The image of any continuous section $\psi$ of $\r \to H_2(\n)$ is a minimal set
of relations:
$$
\n \cong \L(H_1(\n))^\wedge/(\im \psi).
$$

Similarly, if $\n = \bigoplus_{m\ge 1}\n_m$ is a graded Lie algebra that is
generated by $\n_1 \cong H_1(\n)$, it has a presentation of the form
$$
\n \cong \L(H_1(\n))/(\im \psi)
$$
where $\psi : H_2(\n) \to \L(H_1(\n))$ is an injective graded linear mapping.
Observe that a minimal space of relations of degree $n$ is the image of
$$
\psi_n : H_2(\n)_n \to \L_n(H_1(\n)).
$$

\subsection{Chains}

If $\n=\bigoplus_{m\ge 1}\n_m$ is a graded Lie algebra then its
Chevalley-Eilenberg complex is graded. The rows of the diagram
$$
\xymatrix@R=10pt{
& 0 \ar[r] & \Lambda^2 \n_1 \ar[r]^(.55){[\blank,\blank]} & \n_2 \ar[r] & 0 \cr
0 \ar[r] & \Lambda^3 \n_1 \ar[r]^(.45)J & \n_1\otimes \n_2
\ar[r]^(.55){[\blank,\blank]} & \n_3 \ar[r] & 0 
}
$$
are the weight\footnote{Here we use the word ``weight'' to refer the grading
index. Later this index will be the Hodge theoretic weight, which will be
negative.} 2 and 3 components of the Chevalley-Eilenberg chains of $\n$. Here
$J$ is the ``Jacobi identity'' map
\begin{equation}
\label{eqn:jacobi}
J(x\wedge y \wedge z) = x\otimes [y,z] + y\otimes [z,x] + z\otimes [x,y].
\end{equation}
The homology of the $j$th row ($j=1,2$) is the weight $j$ summand $H_\dot(\n)_j$
of $H_\dot(\n)$. The dual complexes compute the weight $j$ summands of
$H^\dot(\n)$. Observe that the subalgebra
$$
\bigoplus_{m\ge 0} H^m(\n)_m
$$
is generated by $H^1(\n)_1 = \Hom(\n_1,F)$. If $\Lambda^2 \n_1 \to \n_2$ is
surjective, then $H^2(\n)_2 = 0$, which implies that $H^m(\n)_m=0$ for all
$m\ge 2$.

Note that when $\n$ is generated by $\n_1$, $H_1(\n)=\n_1$ and the weight $2$
component of $\psi : H_2(\n) \to \L_2(\n_1)= \Lambda^2 \n_1$ is dual to the cup
product $\Lambda^2 H^1(\n) \to H^2(\n)$.

\subsection{Quadratic presentations}

A graded Lie algebra $\n$ is {\em quadratically presented} if it has a graded
presentation of the form
$$
\n = \L(V)/(R)
$$
where $V$ is in weight 1 and $R$ is a subspace of $\L_2(V)$. A pronilpotent Lie
algebra $\n$ with $H_1(\n)$ finite dimensional is {\em quadratically presented}
if it has a presentation of the form
$$
\n = \L(V)^\wedge/(R)
$$
where $R$ is a subspace of $\L_2(V)$. In this case $\n$ is isomorphic to its
degree completion $\prod_{m>0} \Gr^m_\LCS \n$ as a pronilpotent Lie algebra. The
following result is a straightforward consequence of the definitions.

\begin{proposition}
\label{prop:quad_equiv}
Suppose that $\n$ is a pronilpotent Lie algebra with $H_1(\n)$ finite
dimensional. If $\n$ is quadratically presented, then $\Gr^\dot_\LCS \n$ is
quadratically presented and $\n$ is isomorphic (as a pronilpotent Lie algebra)
to the degree completion of $\Gr^\dot_\LCS\n$:
$$
\n \cong \prod_{m\ge 1} \Gr^m_\LCS \n
$$
Conversely, if $\Gr^\dot_\LCS \n$ is quadratically presented and $\n$ is
isomorphic to the degree completion of $\Gr^\dot_\LCS \n$, then $\n$ is
quadratically presented.
\end{proposition}

Quadratically presented graded Lie algebras have a standard characterization. 

\begin{lemma}
\label{lem:quad_pres}
A graded Lie algebra $\n$ that is generated in weight $1$ is quadratically
presented if and only if $H_2(\n)$ has weight $2$ or, equivalently, the cup
product $\Lambda^2 H^1(\n) \to H^2(\n)$ is surjective.
\end{lemma}

\begin{proof}
Suppose that $\n = \L(V)/\r$ is a minimal presentation of $\n$. This means that
$\r\subset [\L,\L]$. Set $\f = \L(V)$. There is a natural isomorphism
$$
H_2(\n) \cong \r/[\f,\r].
$$
A minimal graded space of generators of $\r$ is the image of any section of
$\r\to \r/[\f,\r]=H_2(\n)$. The result follows as $\n$ is quadratically
presented if and only if this set of minimal generators of $\r$ has degree $2$.
\end{proof}

Since morphisms of graded Lie algebras induce graded morphisms of their
homology (and cohomology), we deduce:

\begin{corollary}
\label{cor:quad_pres}
Suppose that $\g \to \n$ is a surjection of graded Lie algebras, both generated
in weight $1$. If $\g$ is quadratically presented and $H_2(\g) \to H_2(\n)$ is
surjective, then $\n$ is quadratically presented.
\end{corollary}

\section{Relative Completion of Discrete Groups}

Suppose that $\G$ is a discrete group and that $R$ is a reductive algebraic
group over a field $F$ of characteristic zero. The completion of $\G$ relative
to a Zariski dense representation $\rho : \G \to R(F)$ is a proalgebraic
$F$-group $\cG$ which is an extension of $R$ by a prounipotent group, and a
homomorphism $\rhohat : \G \to \cG(F)$ such that the composite
$$
\xymatrix{
\G \ar[r]^\rhohat & \cG(F) \ar[r] & R(F)
}
$$
is $\rho$. It is universal for such groups: if $G$ is a proalgebraic $F$ group
that is an extension of $R$ by a prounipotent group, and if $\phi : \G \to G(F)$
is a homomorphism whose composition with $G \to R$ is $\rho$, then there is a
homomorphism $\phihat : \cG \to G$ of proalgebraic $F$-groups such that the
diagram
$$
\xymatrix{
\G \ar[r]^\rhohat\ar[d]_\phi & \cG(F) \ar[dl]_\phihat\ar[d] \cr
G(F) \ar[r] & R(F)
}
$$
commutes.

When $R$ is trivial, $\rho$ is trivial and $\cG = \U$ is the unipotent
completion of $\G$ over $F$.

When discussing the mixed Hodge structure on a relative completion of the
fundamental group of a compact K\"ahler manifold $X$, we need to be able to
compare  the completion of $\pi_1(X,x)$ over $\R$ (or $\Q$) with its completion
over $\C$. For this reason we need to discuss the behaviour of relative
completion under base change.

If $K$ is an extension field of $F$ and $\rho_K : \G \to R(K)$ is Zariski dense
over $K$, then one has the relative completion $\cG_K$ of $\G$ with respect to
$\rho_K$. It is an extension  of $R\times_F K$ by a prounipotent group. The
universal mapping property of $\cG_K$ implies that the homomorphism $\G \to
\cG(K)$ induces a homomorphism $\cG_K \to \cG\times_F K$ of proalgebraic
$K$-groups. This homomorphism is an isomorphism. This is easily seen when each
irreducible representation $V$ of $R$ remains irreducible after extending
scalars from $F$ to $K$. This follows directly from the cohomological results
below.

\subsection{Cohomology}

We continue with the notation above, where $\cG$ is the relative completion of
$\G$. When $R$ is reductive, the structure of $\g$ and $\u$ are closely related
to the cohomology of $\G$ with coefficients in rational representations of $R$.

For each rational representation $V$ of $R$ there are natural isomorphisms
$$
\Hom_R(H_\dot(\u),V) \cong [H^\dot(\u)\otimes V]^R
\cong H^\dot(\g,V)^{\pi_0(R)}.
$$
The homomorphism $\G \to \cG(F)$ induces a homomorphism
\begin{equation}
\label{eqn:homom}
H^\dot(\g,V)^{\pi_0(R)} \to H^\dot(\G,V)
\end{equation}
It is an isomorphism in degrees $\le 1$ and an injection in degree 2.

Denote the set of isomorphism classes of finite dimensional irreducible
representations of $R$ by $\Rdual$. Fix an $R$-module $V_\lambda$ in each
isomorphism class $\lambda \in \Rdual$. If each irreducible representation of
$R$ is absolutely irreducible\footnote{This is the case when $R=\Sp_g$ over any
field of characteristic zero.} and if $H^j(\G,V)$ is finite dimensional for all
rational representations $V$ of $R$ when $j\le 2$, then there is an isomorphism
$$
\prod_{\lambda \in \check{R}} [H^1(\G,V_\lambda)]^\ast \otimes_F V_\lambda
\cong H_1(\u)
$$
of topological modules, and a continuous $R$-invariant surjection
$$
\prod_{\lambda \in \check{R}} [H^2(\G,V_\lambda)]^\ast \otimes_F V_\lambda
\to H_2(\u).
$$
In both cases, the LHS has the product topology.

\begin{lemma}
\label{lem:coho}
Suppose that $\rho : \G \to R(F)$ is a Zariski dense representation from a
discrete group to a reductive group. If $\phi : \G' \to \G$ is a homomorphism
that such that $\rho\circ\phi$ is Zariski dense and such that, for all rational
representations $V$ of $R$,
$$
H^j(\G,V) \to H^j(\G',V)
$$
is an isomorphism when $j=1$ and injective when $j=2$, then the completion
$\cG'\to \cG$ of $\phi$ relative to $\rho$ is an isomorphism. 
\end{lemma}

\begin{proof}
This follows directly from the cohomological results above and that fact that a
homomorphism of pronilpotent Lie algebras $\u'\to\u$ is an isomorphism if and
only if $H^j(\u) \to H^j(\u')$ is an isomorphism in degree 1 and injective in
degree 2.
\end{proof}

\subsection{Hodge theory}
\label{sec:hodge}
Suppose that $X$ is the complement of a normal crossings divisor in a compact
K\"ahler manifold. Suppose that $F=\Q$ or $\R$ and that $\V$ is a polarized
variation of $F$-Hodge structure (PVHS) over $X$. Pick a base point $x_o \in X$.
Denote the fiber over $\V$ over $x_o$ by $V_o$. The Zariski closure of the image
of the monodromy representation
$$
\rho : \pi_1(X,x_o) \to \Aut(V_o)
$$
is a reductive $F$-group. Denote it by $R$. Then one has the relative completion
$\cG$ of $\pi_1(X,x_o)$ with respect to $\rho : \pi_1(X,x_o) \to R(F)$.

\begin{theorem}[\cite{hain:malcev}]
The coordinate ring $\cO(\cG)$ is a Hopf algebra in the category of Ind-mixed
Hodge structures over $F$. It has the property that $W_{-1}\cO(\cG)=0$ and
$W_0\cO(\cG) = \cO(R)$.
\end{theorem}

A slightly weaker version of the theorem is stated in terms of Lie algebras.
Denote the prounipotent radical of $\cG$ by $\U$. Denote their Lie algebras by
$\g$ and $\u$, and the Lie algebra of $R$ by $\r$.

\begin{corollary}[\cite{hain:malcev}]
The Lie algebra $\g$ is a Lie algebra in the category of pro-mixed Hodge
structures over $F$. It has the property that
$$
\g = W_0 \g,\ \u= W_{-1}\g, \text{ and } \Gr^W_0 \g \cong \r.
$$
If $\V$ is a PVHS over $X$ with fiber $V_o$ over the basepoint $x_o$, then the
composite
$$
H^\dot(\g,V_o)^{\pi_0(R)} \to H^\dot(\G,V_o) \to H^\dot(X,\V)
$$
of (\ref{eqn:homom}) with the canonical homomorphism is a morphism of MHS.
\end{corollary}

The existence of the mixed Hodge structure on $\u$ in the unipotent case and
when $X$ is not necessarily compact is due to Morgan \cite{morgan} and Hain
\cite{hain:hodge}.

\subsection{Presentations}
We continue with the notation of Section~\ref{sec:hodge}. The problem of
computing presentations of $\g$ and $\u$ is simplified as the exactness of
$\Gr^W_\dot$ implies that there is a natural isomorphism 
$$
\g \cong \prod_{m\le 0} \Gr^W_m \g
$$
of topological Lie algebras. In order to determine $\g$, it suffices to compute
$\Gr^W_\dot\u$ as a graded Lie algebra in the category of $R$-modules. Exactness
of $\Gr^W_\dot$ implies that $H^\dot(\Gr^W_\dot\u)\cong\Gr^W_\dot H^\dot(\u)$.

When $H_1(\u)$ is finite dimensional, the Lie algebra $\Gr^W_\dot\u$ has a
presentation of the form
$$
\Gr^W_\dot \u \cong \L(\Gr^W_\dot H_1(\u))/(\im \psi)
$$
in the category of Ind $R$-modules for a suitable $R$-module map $\psi :
H_2(\Gr^W_\dot\u) \to L^2 \L(\Gr^W_\dot H_1(\u))$.

The following result is a formal consequence of results stated above and the
fact, due to Deligne (cf.\ \cite[Thm.~2.9]{zucker}), that if $\V$ is a PVHS over
$X$ of weight $m$, then the smooth forms on $X$ with coefficients in $\V$ is a
Hodge complex and $H^j(X,\V)$ has a Hodge structure of weight $m+j$.

\begin{theorem}[{\cite[Thm.~13.14]{hain:malcev}}]
\label{thm:u_quad}
If $X$ is a compact K\"ahler manifold, then
\begin{enumerate}

\item the MHS on $H^j(\u)$ is pure of weight $j$ when $j\le 2$,

\item the weight filtration of $\u$ is its lower central series $W_{-m}\u =
L^m\u$,

\item $\u$ is quadratically presented.

\end{enumerate}
\end{theorem}

In the case of unipotent completion, this was first proved by Deligne,
Griffiths, Morgan and Sullivan in \cite{dgms}.

\section{Symplectic Groups and their Representations}

Suppose that $A$ is a commutative ring and that $H$ is a free $A$-module of rank
$2g$ with a unimodular, skew symmetric bilinear pairing $\bil$. The symplectic
group $\Sp(H)$ is defined to be $\Aut(H,\bil)$. The choice of a symplectic basis
$a_1,\dots,a_g,b_1,\dots,b_g$ of $H$ gives an isomorphism of $\Sp(H)$ with
$\Sp_g(A)$.

When $A$ is a field $F$ of characteristic zero, $\Sp(H)$ is a simply connected
simple algebraic $F$-group all of whose irreducible representations are
absolutely irreducible. Denote its Lie algebra by $\sp(H)$. This is isomorphic
to the Lie algebra $\sp_g(F)$.

Every irreducible representation of $\Sp(H)$ is a submodule of some tensor power
$H^{\otimes m}$ of $H$. Fix a symplectic basis of $H$. Set
$$
\theta := a_1\wedge b_1 + \dots + a_g\wedge b_g \in \Lambda^2 H.
$$
Every linear automorphism of $H$ extends to a derivation of the tensor algebra
$T(H)$ on $H$. With this convention, $\sp(H)$ is the subalgebra of derivations
of $T(H)$ that annihilate $\theta$:
$$
\sp(H) = \{\delta : H \to H : \delta(\theta) = 0\}.
$$
The abelian subalgebra $\h$ of $\sp(H)$ consisting of those derivations
$t=(t_1,\dots,t_g)$ that act on $H$ via
$$
t\cdot a_j = t_j a_j \text{ and } t\cdot b_j = -t_j b_j.
$$
is a Cartan subalgebra. The derivations
$$
S_{i,j} := a_i\partial/\partial a_j - b_j\partial/\partial b_i\ (i< j)
\text { and }
F_{i,j} := a_j \partial/\partial b_i + a_i\partial/\partial b_j (\text{all }i,j)
$$
of $T(H)$ span a nilpotent subalgebra of $\sp(H)$ which, together with $\h$,
span a Borel subalgebra of $\sp(H)$.

A fundamental set of highest weights with respect to this Borel subalgebra is
$\{\lambda_j:1\le j\le g\}$, where $\lambda_j : \h \to F$ is the function
defined by
$$
\lambda_k(t) = t_1 + \dots + t_k.
$$
We will denote the irreducible $\sp(H)$ module with highest weight $\lambda =
\sum_{j=1}^g n_j \lambda_j$ by $V(\lambda)$.

The exterior algebra $\Lambda^\dot H$ is an algebra in the category of
$\Sp(H)$-modules. The fundamental representation $V(\lambda_k)$ is the degree
$k$ part of the quotient of it by the ideal $(\theta)$:
$$
V(\lambda_k) = \Lambda^k_0 H := \Lambda^k H/(\theta\wedge \Lambda^{k-2}H).
$$

Recall that every irreducible $\sp(H)$-module has a non-degenerate invariant
bilinear pairing. Equivalently, every irreducible $\sp(H)$-module is isomorphic
to its dual.

\subsection{Stability in the representation ring of $\Sp(H)$}

Fix a sequence of inclusions
$$
\sp_1 \subset \sp_2 \subset \sp_3 \subset \cdots.
$$
Each of these induces an inclusion $\Rep(\sp_h) \hookrightarrow \Rep(\sp_{h+1})$
of representation rings. One can describe this in terms of symmetric polynomials
(i.e., characters), dominant integral weights or partitions. For example, the
inclusion takes the irreducible representation of $\sp_h$ with highest weight
$\lambda$ to the irreducible representation of $\sp_{h+1}$ with the
corresponding weight. See \cite{kabanov:stability} and \cite[\S6]{hain:torelli}
for details. In \cite{kabanov:stability}, Kabanov proves that the decomposition
of Schur functors and tensor products of representations stabilize in this
sense. Stability allows one to compute stable decompositions of tensor products
and Schur functors, by computing the decomposition in one case in the stable
range. See \cite[\S 6]{hain:torelli} for a more detailed discussion.

\subsection{Unipotent completion of surface groups}

Suppose that $C$ is a compact Riemann surface of genus $g$ and that $x\in C$.
Denote the Lie algebra of the unipotent completion $\cP$ of $\pi_1(C,x)$ by
$\p$. Set $H=H_1(C;\Q)$. Let $a_1,\dots,a_g,b_1,\dots,b_g$ be a symplectic basis
of $H$. Applying the results of the previous section, $\p$ has a natural MHS
whose associated graded has the presentation:
$$
\Gr^W_\dot \p \cong \L(H)/(\scriptstyle{\sum_{j=1}^g} [a_j, b_j]).
$$
This is a Lie algebra in the category of $\Sp(H)$-modules.\footnote{This
statement is also a consequence of Labute's Theorem \cite{labute}.}

The $g\ge 3$ case of the following result is \cite[Prop.~8.4]{hain:torelli}. The
genus $2$ case is proved similarly.

\begin{proposition}
The highest weight decomposition of the first four weight graded quotients of
$\p$ are:
$$
\Gr^W_{-m} \p =
\begin{cases}
V(\lambda_1) & m=1,\ g\ge 1,\cr
V(\lambda_2) & m=2,\ g\ge 2,\cr
V(\lambda_1+\lambda_2)& m=3,\ g\ge 2,\cr
V(\lambda_1+\lambda_3) + V(2\lambda_1) + V(2\lambda_1+\lambda_2) & m=4,\ g\ge 3.
\end{cases}
$$
\end{proposition}

The exactness of $\Gr^W_\dot$ implies that there are natural graded Lie algebra
isomorphisms
$$
\Gr^W_\dot\Der\p \cong \Der\Gr^W_\dot \p
\text{ and }
\Gr^W_\dot\OutDer\p \cong \OutDer\Gr^W_\dot \p.
$$
Since $\Gr^W_\dot\p$ (and hence $\p$ as well) has trivial center
\cite{asada-kaneko}, the sequence
$$
0 \to \Gr^W_\dot\p \to \Der \p \to \OutDer\p \to 0
$$
is exact.

\begin{corollary}[{\cite[Cor.~9.4]{hain:torelli}}]
\label{cor:outder}
When $g\ge 3$, we have
$$
\Gr^W_{-m} \OutDer \p =
\begin{cases}
V(\lambda_3) & m=1,\cr
V(2\lambda_2) & m=2,\cr
V(2\lambda_1+\lambda_3) + V(3\lambda_1) & m=3.
\end{cases}
$$
\qed
\end{corollary}

\section{Completions of Mapping Class Groups}

In this section $g$ and $n$ are non-negative integers satisfying $2g-2+n>0$.

\subsection{Relative completion of mapping class groups}

Let $S$ be a closed oriented surface of genus $g$. Set $H=H_1(S;\Z)$. Fix a
symplectic basis of $H$. The action of $\G_{g,n}$ on $S$ induces a surjective
homomorphism
$$
\rho : \G_{g,n} \to \Aut(H,\bil) \cong \Sp_g(\Z).
$$
Since $\Sp_g(\Z)$ is Zariski dense in $\Sp_g(\Q)$, we have the completion
$\cG_{g,n}$ of $\G_{g,n}$ with respect to $\rho$. It is an extension
$$
1 \to \U_{g,n} \to \cG_{g,n} \to \Sp_g \to 1
$$
of the symplectic group by a prounipotent group. Denote the Lie algebras of
$\cG_{g,n}$ and $\U_{g,n}$ by $\g_{g,n}$ and $\u_{g,n}$, respectively. When
$n=0$, it will be suppressed, so that $\cG_g = \cG_{g,0}$, $\g_g = \g_{g,0}$,
etc.

When $m>0$, the {\em level $m$ subgroup} $\G_{g,n}[m]$ of $\G_{g,n}$ is defined
to be the kernel of the mod $m$ reduction $\rho_m : \G_{g,n} \to \Sp_g(\Z/m)$ of
$\rho$. When $g\ge 3$, the homomorphism $\G_{g,n}[m] \to \cG_{g,n}(\Q)$ is the
completion of $\G_{g,n}$ relative to $\rho : \G_{g,n} \to \Sp_g(\Q)$,
\cite[Prop.~3.3]{hain:torelli}. In Section~\ref{sec:putman} we will see that a
recent result of Putman \cite{putman} implies that this result extends to finite
index subgroups of mapping class groups that contain its Johnson subgroup,
provided that $g\ge 3$. When $g\le 2$ the completion of $\G_{g,n}[m]$ depends
non-trivially on $m$. 

The homomorphism $T_{g,n}\to \U_{g,n}(\Q)$ induces a homomorphism $\T_{g,n} \to
\U_{g,n}$ where $\T_{g,n}$ denotes the unipotent completion of $T_{g,n}$.

\begin{theorem}[\cite{hain:comp},{\cite[Thm.~3.4]{hain:torelli}}]
\label{thm:cent_ext}
If $g\ge 2$, then this homomorphism is surjective. If $g\ge 3$, it has
non-trivial kernel isomorphic to the additive group $\Ga$. The sequence
$$
0 \to \Ga \to \T_{g,n} \to \U_{g,n} \to 1
$$
is a non-trivial central extension.
\end{theorem}

The action of $\G_{g,1}$ on $\cP$ induces an action $\G_{g,1} \to \Aut \p$ that
preserves the lower central series filtration. The action of $\G_{g,1}$ on
$\Gr^\dot_\LCS\p$ factors through $\rho : \G_{g,1} \to \Sp(H)$. The universal
mapping property of relative completion, implies that this action induces an
action $\cG_{g,1}\to \Aut \p$. The induced action $\g_{g,1} \to \Der\p$ on Lie
algebras is a tool for understanding $\u_{g,1}$. It induces the outer action
$\g_g \to \OutDer \p$.

\subsection{Mapping class groups as fundamental groups of smooth varieties}

The moduli space $\M_{g,n}$ of smooth complex projective curves of type $(g,n)$
is the quotient of the Teichm\"uller space $\X_{g,n}$ by $\G_{g,n}$, which acts
properly discontinuously on Teichm\"uller space. Although this action is not
fixed point free, it is when restricted to $\G_{g,n}[m]$ when $m\ge 3$. The
quotient $\M_{g,n}[m] := \G_{g,n}[m]\bs \X_{g,n}$ is a smooth quasi projective
variety. Consequently, $\G_{g,n}[m]$ is the fundamental group of a smooth
quasi-projective variety for all $m\ge 3$.

One can realize $\G_{g,n}$ as the fundamental group of a smooth quasi-projective
variety $X$ by fixing $m\ge 3$ and taking
$$
X = \M_{g,n}[m]\times_{\Sp_g(\Z/m)} Y
$$
where $Y$ is a simply connected projective manifold on which $\Sp_g(\Z/m)$ acts
fixed point freely.\footnote{Such varieties $Y$ were constructed by Serre. The
argument can be found in \cite[IX,\S4.2]{shafarevich}.} Alternatively, when
$g+n>3$ and $g\ge 2$, one can take $X$ to be the complement of the locus in
$\M_{g,n}$ of $n$-pointed curves with a non-trivial automorphism (cf.\
\cite[Prop.~4.1]{hain:jams}).

Define $\Ghat_{g,n}[m]$ to be the orbifold fundamental group of $\cC_g^n[m]$,
the $n$th power of the universal curve over $\M_g[m]$. It is an extension
$$
1 \to \pi^n \to \Ghat_{g,n}[m] \to \G_g[m] \to 1,
$$
where $\pi$ denotes the fundamental group of a smooth projective curve of genus
$g$. Since $\M_{g,n}[m]$ is a Zariski open subset of $\cC_g^n[m]$,
$\Ghat_{g,n}[m]$ is a quotient of $\G_{g,n}[m]$.

Since $\cC_g^n[m]$ is a smooth quasi-projective variety when $m\ge 3$,
$\Ghat_{g,n}[m]$ is the fundamental group of a smooth quasi-projective variety
for all $m \ge 1$. (Use the trick above.)

\subsubsection{Associated Hodge theory}

The fact that mapping class groups are fundamental groups of smooth varieties
allows us to study their relative completions via Hodge theory.

Suppose that $\G$ is one of the mapping class groups $\G_{g,n}$ or
$\Ghat_{g,n}$. Then $\G$ is the fundamental group of a pointed smooth variety
$(\M,x)$, where $\M$ is one of the smooth varieties described above. Denote its
completion relative to the homomorphism $\rho : \G \to \Sp(H)$ by $\cG$ and it
prounipotent radical by $\U$. One has a family $f:\cC \to \M$ of complete genus
$g$ curves. The representation $\rho$ is the monodromy representation of the
local system $R^1f_\ast \Q$, which is a polarized variation of Hodge structure.
The Lie algebras $\g$ of $\cG$ and $\u$ of $\U$ are Lie algebras in the category
of pro-mixed Hodge structures.

The Torelli subgroup $T$ of $\G$ is the kernel of $\rho : \G \to \Sp(H)$. Denote
the Lie algebra of its unipotent completion by $\t$.

The following result summarizes several results from Sections~3 and 4 of
\cite{hain:torelli}.

\begin{proposition}
\label{prop:hodge_mcg}
If $g\ge 3$, then
\begin{enumerate}

\item the MHS on $\u$ can be lifted to a MHS on $\t$ so that the central
extension
\begin{equation}
\label{eqn:cent_extn}
0 \to \Q(1) \to \t \to \u \to 0
\end{equation}
is a non-trivial central extension of Lie algebras in the category of pro-MHS.
Consequently, the associated graded sequence
$$
0 \to \Q(1) \to \Gr^W_\dot \t \to \Gr^W_\dot \u \to 0
$$
is an exact sequence of graded Lie algebras in the category of $\Sp(H)$ modules.

\item the weight filtrations of $\u$ and $\t$ are their lower central series:
$$
W_{-m} \u = L^m \u \text{ and } W_{-m} \t = L^m \t.
$$ 
In particular, $H_1(\u) =  H_1(\t) = \Gr^W_{-1}\u = \Gr^W_{-1}\t$.

\item $H_1(\u_{g,n}) \cong V(\lambda_3) + V(\lambda_1)^n$.

\end{enumerate}
\end{proposition}

The spectral sequence of the central extension gives a Gysin sequence.

\begin{corollary}
\label{cor:gysin}
For all $g\ge 3$ there is a Gysin sequence
\begin{multline*}
0 \to \Q(-1) \to H^2(\u_g) \to H^2(\t_g) \to H^1(\u_g)(-1)\cr
 \to H^3(\u_g) \to H^3(\t_g) \to H^2(\u_g)(-1) \to \dots
\end{multline*}
in the category of mixed Hodge structures. It remains exact after applying
$\Gr^W_m$ for all $m\in \Z$.
\end{corollary}

\subsection{Variation on a theme of Putman}
\label{sec:putman}

Here we recall a result of Putman and extract some useful consequences from it.
As in the introduction, $K_{g,n}$ denotes the subgroup of $T_{g,n}$ generated by
Dehn twists on bounding simple closed curves (BSCCs).\footnote{Recall that a
{\em bounding pair} in a surface $S$ consists of two disjoint non-separating
simple closed curves $B_+$ and $B_-$ which together divide the surface into two
components. A {\em bounding pair map} (BP map) consists of the product
$\smash{t_{B_+}t_{B_-}^{-1}}$ of a positive Dehn twist about $B_+$ and a
negative Dehn twist about $B_-$.} The following result is a special case of a
result \cite[Thm.~A]{putman} of Putman.

\begin{theorem}[Putman]
Suppose that $g\ge 3$ and that $n\ge 0$. If $S$ is a finite index subgroup of
$T_{g,n}$ that contains $K_{g,n}$, then the inclusion induces an isomorphism
$H_1(S,\Q) \to H_1(T_{g,n},\Q)$.
\end{theorem}

Denote the image of $K_{g,n}$ in  $\Ghat_{g,n}$ by $\Khat_{g,n}$.

\begin{corollary}
Suppose that $g\ge 3$ and that $n\ge 0$. If $S$ is a finite index subgroup of
$\That_{g,n}$ that contains $\Khat_{g,n}$, then the inclusion induces an
isomorphism $H_1(S,\Q) \to H_1(\That_{g,n},\Q)$.
\end{corollary}

\begin{proof}
Observe that $\That_{g,n}$ is the fundamental group of the $n$th power of the
universal curve over Torelli space $BT_g$ and that $T_{g,n}$ is the fundamental
group of the Zariski open subset obtained by removing the ``diagonals''. It
follows that $T_{g,n}\to \That_{g,n}$ is surjective. Consequently, the inverse
image $S'$ of $S$ in $T_{g,n}$ is a finite index subgroup that contains
$K_{g,n}$. Now consider the diagram:
$$
\xymatrix{
S' \ar[r]\ar[d] & T_{g,n} \ar[d] \cr
S \ar[r] & \That_{g,n}
}
$$
Since the left-hand vertical arrow is surjective, it induces a surjection on
$H_1$.  Putman's result implies that top map induces an isomorphism on rational
$H_1$. The result follows as the right-hand vertical map induces an isomorphism
on rational $H_1$, which follows from \cite[Prop.~5.2]{hain:normal} and an easy
spectral sequence argument for $n$th power of the universal curve over $\T_g$.
\end{proof}

Putman's result has the following analogue for mapping class groups.

\begin{proposition}
\label{prop:putman_mcg}
Suppose that $g\ge 3$ and that $n\ge 0$. If $\G$ is a finite index subgroup of
$\G_{g,n}$ (resp.\ $\Ghat_{g,n}$) that contains $K_{g,n}$ (resp.\
$\Khat_{g,n}$), then for all $\Sp_g$-modules $V$, the map
$$
H^\dot(\G_{g,n},V) \to H^\dot(\G,V)
$$
induced by the inclusion $\G \to \G_{g,n}$ is an isomorphism in degrees $\le 1$
and an injection in degree $\ge 2$. In particular,
$$
H^1(\G,V(\lambda) \cong
\begin{cases}
\Q &  \lambda = \lambda_3,\cr
\Q^n & \lambda = \lambda_1,\cr
0 & \text{otherwise.}
\end{cases}
$$
\end{proposition}

As in \cite[\S5]{hain:normal}, the vanishing of $H^1(\G,\Q)$ implies that the
Picard group of the corresponding moduli space of curves is finitely generated.

\begin{proof}
We will prove the $\G_{g,n}$ case; the $\Ghat_{g,n}$ case is similar and left to
the reader. Since $V$ is a $\Q$-module, a standard trace argument implies that
the induced map on cohomology is injective in all degrees. Surjectivity in
degree 0 follows from the fact that the image of $\G$ in $\Sp_g(\Q)$ has finite
index in $\Sp_g(\Z)$ and is thus Zariski dense.

To prove that the induced map on cohomology is an isomorphism in degree 1,
write $\G$ as an extension
$$
1 \to \G\cap T_{g,n} \to \G \to L \to 1,
$$
where $L = \im\{\G \to \Sp_g(\Z)\}$. The kernel is a finite index subgroup $S$
of $T_{g,n}$ that contains $K_{g,n}$ and so Putman's theorem implies that
$H^1(S,\Q)=H^1(T_{g,n},\Q)$. Since $g\ge 3$, Raghunathan's vanishing theorem
\cite{raghunathan} implies the vanishing of $H^1(L,V)$, and thus the injectivity
of the restriction map
$$
H^1(\G,V) \to H^0(L,H^1(S)\otimes V) = \Hom_L(H_1(S),V).
$$
On the other hand, there is an isomorphism \cite[p.113]{hain:normal}
$$
H^1(\G_{g,n},V) \to H^0(\Sp_g(\Z),H^1(T_{g,n})\otimes V)
= \Hom_{\Sp_g}(H_1(T_{g,n}),V)
$$
Putman's theorem implies that the restriction mapping $H^1(\G_{g,n},V) \to
H^1(\G,V)$ is an isomorphism.

The last assertion follows from the computation of $H^1(\G_{g,n},V(\lambda))$ in
\cite[\S5]{hain:normal}.
\end{proof}

Combining Putman's theorem and its variants with Lemma~\ref{lem:coho} yields: 

\begin{corollary}
\label{cor:isom_comp}
Suppose that $g\ge 3$ and $n\ge 0$.
\begin{enumerate}

\item If $S$ is a finite index subgroup of $T_{g,n}$ that contains $K_{g,n}$,
then the inclusion of $S$ into $T_{g,n}$ induces an isomorphism on unipotent
completions.

\item If $\G$ is a finite index subgroup of $\G_{g,n}$ that contains $K_{g,n}$,
then the inclusion $\G \to \G_{g,n}$ induces an isomorphism on completions
relative to the standard representation $\rho$.

\item If $S$ is a finite index subgroup of $\That_{g,n}$ that contains
$\Khat_{g,n}$, then the inclusion of $S$ into $T_{g,n}$ induces an isomorphism
on unipotent completions.

\item If $\G$ is a finite index subgroup of $\Ghat_{g,n}$ that contains
$\Khat_{g,n}$, then the inclusion $\G \to \Ghat_{g,n}$ induces an isomorphism on
completions relative to the standard representation $\rho$.

\end{enumerate}
\end{corollary}

There are interesting and natural examples of moduli spaces whose fundamental
groups do not contain the Torelli group, but do contain its Johnson subgroup.

\begin{proposition}
When $g\ge 3$,
\begin{enumerate}

\item (Boggi \cite{boggi}) the fundamental group of the ``Prym-level'' moduli
spaces  \cite{looijenga,boggi-pikaart} (these are called the ``Prym-level
mapping class groups'') of type $(g,n)$ contain $K_{g,n}$,

\item the fundamental group of the moduli space of genus $g$ curves with
a $(2g-2)$nd root of their canonical bundle contains $K_g$.

\end{enumerate}
\end{proposition}

\begin{proof}[Sketch of Proof]
Let $(S;x_1,\dots,x_n)$ be an $n$-pointed reference surface of genus $g$, where
$S$ is compact. Set $S'= S-\{x_1,\dots,x_n\}$. The Prym-level mapping subgroups
of $\G_{g,n}$ are kernels of homomorphisms $\G_{g,n} \to \Out G$, where $G$ is a
finite quotient of $\pi_1(S',x)/W_{-3}$ and the weight filtration is pulled back
from that of the unipotent completion of $\pi_1(S',x)$. Since $K_{g,n}$ acts
trivially on $\pi_1(S',x)/W_{-3}$, Prym-level mapping class groups of type
$(g,n)$ contain $K_{g,n}$. See \cite[p.~179]{boggi-pikaart} for definitions.

As explained in \cite[\S 13]{hain:normal},\footnote{There is a typo in the first
display of Theorem~13.3: it should be
$$
0 \to 2 H_1(C,\Z/n\Z) \to {\mathcal K}_n^{(2)} \overset{\pi}{\to}
\Sp_g(\Z/n\Z) \to 1.
$$
\vspace{-3ex}
}
a result of Sipe \cite{sipe} implies that when $g\ge 3$ the action of the
Torelli group $T_g$ on the $(2g-2)$nd roots of the canonical bundle factors
through the Johnson homomorphism $T_g \to \Lambda^3 H_\Z/(\theta\wedge H_\Z)$.
This implies that $K_g$, and hence all $K_{g,n}$, act trivially on roots of the
canonical bundle when $g\ge 3$.
\end{proof}

\subsection{Kabanov purity}
\label{sec:kabanov}

Kabanov \cite{kabanov:purity} proved that the natural map from $H^2(\M_g,\V)$ to
the degree two intersection cohomology of the Satake compactification of $\M_g$
is an isomorphism when $g\ge 6$. When $3 \le g < 6$ he proved a more technical
statement. These statements, via the arguments in \cite[\S7]{hain:torelli},
imply the following result, which bounds the weights of the relations in
a presentation of $\Gr^W_\dot \u_{g,n}$.

\begin{theorem}
\label{thm:kabanov}
If $g\ge 6$, then $\Gr^W_m H^2(\u_{g,n}) = 0$ when $m\neq 2$. When $3 \le g <
6$, $\Gr^W_m H^2(\u_{g,n}) = 0$ when $m\notin \{2,3\}$. \qed
\end{theorem}

This implies that the minimal presentations of $\u_{g,n}$ has quadratic
relations when $g\ge 6$ and at worst cubic relations when $3\le g<6$.

\section{Presentations of $\u_g$ and $\t_g$ for $g\ge 4$}
\label{sec:quad_pres}

In this section we prove that $\u_{g,n}$, $\uhat_{g,n}$, $\t_{g,n}$ and
$\that_{g,n}$ are quadratically presented when $g\ge 4$ and $n\ge 0$. This was
proved in \cite{hain:torelli} when $g\ge 6$ using Kabanov's Purity Theorem. Here
we combine Sakasai's computations \cite{sakasai} with Kabanov purity to extend
this to the cases $g=4,5$. The same tools allow us to bound the relations in
genus $3$.

The starting point is to recall that $\Gr^W_m \t_g = \Gr^W_m \u_g$ except when
$m=-2$, where they differ by a copy of the trivial representation. To compute
$\Gr^W_{-3}\u_g$, we consider the homomorphism
\begin{equation}
\label{eqn:outer_action}
\Gr^W_\dot\u_g \to \OutDer\Gr^W_\dot\p.
\end{equation}
Johnson's fundamental work \cite{johnson:h1} implies that it is an isomorphism
in weight $-1$ when $g\ge 3$, in which case all groups are isomorphic to
$V(\lambda_3)$.\footnote{We could equally well have used the action
$\Gr^W_\dot\u_{g,1} \to \Der\Gr^W_\dot\p$. All results in the sequel have
equivalent analogues for it. This is because $\p$ has trivial center
\cite{asada-kaneko}, which implies that $\ker\{\u_{g,1}\to \u_g\}\cong
\ker\{\Der\p\to\OutDer\p\}$.}

\begin{proposition}[{\cite[Thm.~10.1]{hain:torelli}}]
\label{prop:wt_2}
The homomorphism (\ref{eqn:outer_action}) is an isomorphism in weight $-2$ for
all $g\ge 3$. In this case $\Gr^W_{-2}\u_g \cong V(2\lambda_2)$.
\end{proposition}

In this section we will show that (\ref{eqn:outer_action}) is injective in
weight $-3$ for all $g\ge 4$. In Section~\ref{sec:pres_3} we will show that it
also injective in weight $-3$ when $g=3$, which will allow us to compute a
minimal set of cubic relations in genus $3$.

It is not known whether the outer action (\ref{eqn:outer_action}) is injective
in any genus. However, we believe (without much evidence) that it is ``stably
injective''.

\begin{question}
Is there a monotonic unbounded function $w: \N \to \N$ such that
$$
\Gr^W_{-m} \u_{g,1} \to \Gr^W_{-m} \p
$$
is injective when $m \le w(g)$?
\end{question}

This is a weaker version of Question 9.7 in \cite{hain-looijenga}.

\subsection{Quadratic relations}
\label{sec:quad_relns}

Let $\n$ be $\u_g$ or $\t_g$ when $g\ge 3$.
Since $H^1(\n)$ has weight $1$, $H^1(\n) = (\Gr^W_{-1}\n)^\ast$, the cup
product and the quadratic relations in $\Gr^W_\dot\n$ are related by the exact
sequence
$$
\xymatrix{
0 \ar[r] & \big(\Gr^W_{-2}\n\big)^\ast \ar[r]^\bracket &
\Lambda^2 H^1(\n) \ar[r]^(.45){\text{cup}} & \Gr^W_2H^2(\n) \ar[r] & 0,
}
$$
which is the weight $2$ summand of the Chevalley-Eilenberg cochains of $\n$.

\begin{proposition}[{\cite[Lem.~10.2]{hain:torelli}}]
As an $\Sp_g$-module, the second exterior power of $V:=V(\lambda_3)$ decomposes:
$$
\L_2(V) =
\Lambda^2 V =
\begin{cases}
V(0)+V(2\lambda_2) & g=3,\cr
V(0)+V(2\lambda_2) + V(\lambda_2+\lambda_4) + V(\lambda_2) & g=4,\cr
V(0)+V(2\lambda_2) + V(\lambda_2+\lambda_4) + V(\lambda_2) + V(\lambda_4)
& g=5,\cr
V(0)+V(2\lambda_2) + V(\lambda_2+\lambda_4) + V(\lambda_2) + V(\lambda_4)
+ V(\lambda_6) & g\ge 6.
\end{cases}
$$
\qed
\end{proposition}

Since the image of the cup product $\Lambda^2 H^1(\t_g)\to H^2(\t_g)$
is dual to the set of quadratic relations in $\Gr^W_\dot\t_g$, we have:

\begin{corollary}[Cf.\ {\cite[Thm.~10.1]{hain:torelli}}]
\label{cor:cup_prod}
The space of quadratic relations $R$ in a minimal presentation of $\Gr^W_\dot
\t_g$ are dual to (and hence isomorphic to) the image of the cup product:
$$
\im\{\Lambda^2 H^1(T_g,\Q) \to H^2(T_g,\Q)\}=
\begin{cases}
0 & g=3, \cr
V(\lambda_2+\lambda_4) + V(\lambda_2) & g=4, \cr
V(\lambda_2+\lambda_4) + V(\lambda_2) + V(\lambda_4) & g=5,\cr
V(\lambda_2+\lambda_4) + V(\lambda_2) + V(\lambda_4) + V(\lambda_6) & g\ge 6.
\end{cases}
$$
In particular, there are no quadratic relations in $\Gr^W_\dot \t_3$. The space
of quadratic relations in $\Gr^W_\dot \u_g$ is $R\oplus V(0)$.
\end{corollary}

As explained in \cite[\S10]{hain:torelli}, a generic vector in the space of
quadratic relations is given by a pair of disjoint bounding pair maps.

\subsection{Sakasai's computation}

The following assertion is the dual of \cite[Prop.~6.2]{sakasai}.

\begin{proposition}[Sakasai]
\label{prop:sakasai}
The cokernel of the Jacobi identity map (\ref{eqn:jacobi})
$$
J : \Lambda^3 \Gr^W_{-1}\t_g \to \Gr^W_{-1}\t_g \otimes \Gr^W_{-2}\t_g
$$
has highest weight decomposition
$$
\coker J \cong
\begin{cases}
V(2\lambda_1 + \lambda_3) + V(\lambda_1+\lambda_2) + V(2\lambda_2+\lambda_3)
+ V(\lambda_3) & g=3,\cr
V(2\lambda_1 + \lambda_3) & g \ge 4.
\end{cases}
$$ \qed
\end{proposition}

\begin{corollary}
\label{cor:sakasai}
For all $g\ge 3$, the image of $\Gr^W_{-3}\u_g \to \Gr^W_{-3}\OutDer \p$ is
isomorphic to $V(2\lambda_1 + \lambda_3)$. If $g=3$, then
\begin{enumerate}

\item one has the exact sequence
$$
0 \to \Gr^W_{-3} H_2(\t_3) \to
V(2\lambda_1 + \lambda_3) + V(\lambda_1+\lambda_2) + V(2\lambda_2+\lambda_3)
+ V(\lambda_3) \to \Gr^W_{-3}\t_3 \to 0
$$
of $\Sp(H)$-modules.

\item the sequence
$$
0 \to \Q(-1) \to H^2(\u_3) \to H^2(\t_3) \to V(-1) \to 0
$$
is exact. In particular, $\Gr^W_{-m} H^2(\t_3) = 0$ when $m\neq 3$.
Here $V=V(\lambda_3)$, which is regarded as having weight
$1$.

\end{enumerate}
If $g\ge 4$, $\Gr^W_m H^2(\t_g)$ vanishes when $m\neq 2$, $\Gr^W_{-3}\t_g =
V(2\lambda_1+\lambda_3)$, and the sequence
$$
0 \to \Q(-1) \to H^2(\u_g) \to H^2(\t_g) \to 0
$$
is exact. Consequently, $\Gr^W_m H^2(\u_g)$ vanishes when $m\neq 2$.
\end{corollary}

Combined with Kabanov purity, this implies that $\t_3$ has only cubic relations
and that $\t_g$ and $\u_g$ are quadratically presented for all $g\ge 4$ as
$H^2(\t_g)$ and $H^2(\u_g)$ are pure of weight 2

\begin{proof}

Corollary~\ref{cor:outder} provides a lower bound for $\Gr^W_{-3}\t_g$. Since
$V(3\lambda_1)$ does not occur in $\L_3(V(\lambda_3))$, the image of
$\Gr^W_{-3}\t_g$ in $\Gr^W_{-3}\Out\Der\p$ must be $V(2\lambda_1+\lambda_3)$ or
trivial. If it were trivial, the image of $\t_g$ in $\OutDer\p$ would be
nilpotent, which is not the case. So the image of $\Gr^W_{-3}\t_g$ must be
$V(2\lambda_1+\lambda_3)$.

The homology of the sequence
$$
\xymatrix{
0 \ar[r] & \Lambda^3 \Gr^W_{-1}\t_g \ar[r]^(0.4)J &
\Gr^W_{-1}\t_g \otimes \Gr^W_{-2}\t_g
\ar[r]^(.55){\text{bracket}} & \Gr^W_{-3}\t_g \to 0
}
$$
is $\Gr^W_{-3}H_\dot(\t_g)$. Here the left-hand term is placed in degree 3 and
the right-hand term in degree 1. Since $\Gr^W_{-3}H_1(\t_g) = 0$, the right-hand
map is surjective. This implies that the sequence
$$
0 \to \Gr^W_{-3}H_2(\t_g) \to \coker J \to \Gr^W_{-3}\t_g \to 0
$$
is exact. When $g=3$, this implies assertion (i). When $g\ge 4$, Sakasai's
computation combined with the first assertion implies that $\Gr^W_{-3}H_2(\t_g)$
vanishes and that $\Gr^W_{-3}\t_g = V(2\lambda_1 + \lambda_3)$.

Since $H^1(\u_g) \to H^1(\t_g)$ is an isomorphism, the Gysin sequence
Corollary~\ref{cor:gysin} gives an exact sequence
$$
\xymatrix{
0 \ar[r] & \Q(-1) \ar[r] & H^2(\u_g) \ar[r]^\phi &
H^2(\t_g) \ar[r]^\psi & V(-1)
}
$$
of MHS in the category of $\Sp(H)$-modules. Kabanov purity implies $\Gr^W_m
H^2(\u_g)=0$ when $m>3$ and $g\ge 3$. Since $V(-1)$ has weight $3$, this implies
that $\Gr^W_m H^2(\t_g)$ also vanishes when $g\ge 3$ and $m>3$.

When $g\ge 4$, the vanishing of $\Gr^W_3 H^2(\t_g)$ implies the exactness of
$$
0 \to \Q(-1) \to H^2(\u_g) \to H^2(\t_g) \to 0.
$$
Combined with Kabanov purity, this implies that $\Gr^W_m H^2(\u)$ and $\Gr^W_m
H^2(\t_g)$ vanish for all $m\ge 3$.

To complete the proof, we prove assertion (ii). Since $\t_3$ has no quadratic
relations, and since the copy of the trivial representation in $\Gr^W_{-2}\t_3$
is central, there is a copy of $V$ in the cubic relations. This relation cannot
be a consequence of quadratic relations as there are none. This implies that the
homomorphism $\psi : H^2(\t_3) \to V(-1)$ is surjective. Plugging this into the
Gysin sequence above establishes (ii). 
\end{proof}

\subsection{Proof of Theorem~\ref{thm:torelli}}

By Proposition~\ref{prop:quad_equiv} and Lemma~\ref{lem:quad_pres} we only need
show that $H^2(\t_{g,n})$ and $H^2(\u_{g,n})$ are pure of weight 2. When $n=0$,
we know this from Corollary~\ref{cor:sakasai}. This yields a quadratic
presentation with the relations computed in Corollary~\ref{cor:cup_prod}:

\begin{theorem}
The Lie algebras $\u_g$ and $\t_g$ are quadratically presented for all $g\ge 4$.
The presentations are
$$
\t_g = \L(\Lambda^3_0 H)^\wedge/(R) \text{ and }
\u_g = \L(\Lambda^3_0 H)^\wedge/(R + V(0))
$$
where $R$ is the orthogonal complement of $V(2\lambda_2) + V(0)$ in $\Lambda^2_0
\Lambda^3_0 H = \L_2(\Lambda^3_0 H)$ described explicitly in
Corollary~\ref{cor:cup_prod}.
\end{theorem}

\subsubsection{The case $n>0$}

The degree 2 cohomology of each of the Lie algebras $\u_{g,n}$, $\uhat_{g,n}$,
$\t_{g,n}$ and $\that_{g,n}$ is also pure of weight 2 when $n\ge 0$ and $g\ge
4$. This can be proved using the fact that, in each case,
\begin{enumerate}

\item the result is true for each series of Lie algebras when $n=0$,

\item $H^1$ of each of these Lie algebras is pure of weight $1$, and

\item  $H^2$ is pure of weight $2$ if and only it is true in the base case,
$n=0$.

\end{enumerate}
This can be proved by an easy spectral sequence applied, for example, to the
extension
$$
0 \to \p^n \to \uhat_{g,n} \to \u_g \to 0.
$$
Since $H^2(\u_g)$ is of weight 2 for all $g\ge 4$, this establishes the
assertion. Details are left to the reader.

As explained above, this result combined with Proposition~\ref{prop:quad_equiv}
and Lemma~\ref{lem:quad_pres} imply the following result, of which
Theorem~\ref{thm:torelli} is a special case.

\begin{theorem}
\label{thm:genus>3}
The Lie algebras $\u_{g,n}$, $\uhat_{g,n}$ and $\t_{g,n}$ and $\that_{g,n}$ are
quadratically presented for all $g\ge 4$ and all $n\ge 0$. \qed
\end{theorem}

\subsection{Proof of Theorem~\ref{thm:main_torelli}}

Recall from Section~\ref{sec:putman} that $K_{g,n}$ is the subgroup of $T_{g,n}$
generated by Dehn twists on bounding simple closed curves and that $\Khat_{g,n}$
is its image in $\That_{g,n}$. A group homomorphism $G \to H$ is {\em virtually
surjective} if its image has finite index in $H$.

\begin{theorem}
Suppose that $\G$ is a finitely generated group and that  $\G \to T_3$ is a
virtually surjective homomorphism. If $H_2(\G,\Q) \to H_2(T_3,\Q)$ is
surjective, then $\G$ is not a K\"ahler group. In particular, no finite index
subgroup of $T_{3,n}$ that contains $K_{3,n}$ and no finite index subgroup of
$\That_{g,n}$ that contains $\Khat_{3,n}$ is a K\"ahler group.
\end{theorem}

\begin{proof}
From \cite[Rem.~10.4]{hain:torelli} (or Cor.~\ref{cor:sakasai}) we know that
$\Gr^W_3 H^2(\t_3)\supseteq V(\lambda_3)$ and is therefore non-zero.
Proposition~\ref{prop:quad_equiv} and Lemma~\ref{lem:quad_pres} imply that
$\t_3$ is not quadratically presented. The remaining cases follow by applying
Corollary~\ref{cor:quad_pres} to the homomorphism $\g \to \t_3$ from the Lie
algebra of the unipotent completion of $\G$ to $\t_3$.
\end{proof}

This completes the proof of Theorem~\ref{thm:main_torelli}. In the proof of
Theorem~\ref{thm:main_mcg}, which is given in Section~\ref{sec:kahler}, we will
need to know that $\u_3$ is not quadratically presented. This does not follow
from the fact that $\t_3$ is not quadratically presented. It is conceivable that
$\u_3$ is quadratically presented even though $\t_3$ is not. This would happen
if $\Gr^W_3 H^2(\t_3)$ were  equal to $V(\lambda_3)$. In that case $\u_3$ would
be isomorphic to $\p_7$, the Lie algebra associated to the unipotent fundamental
group of a genus 7 surface $S$, and $\t_3$ would be isomorphic to the Lie
algebra $\p_{7,\vec{1}}$ of the unipotent fundamental group of the unit tangent
bundle of $S$. These are related by a central extension
$$
0 \to \Q(1) \to \p_{7,\vec{1}} \to \p_7 \to 0.
$$
To prove that $\u_3$ is not quadratically presented we will show that $\Gr^W_3
H^2(\t_3)$ is strictly larger than $V(\lambda_3)$. To achieve this we compute
explicit presentations of $\t_3$ and $\u_3$ in Section~\ref{sec:pres_3}.

\section{Commuting Elements}
\label{sec:commuting}

In this section, we assume the reader is familiar with the Johnson homomorphism.
Surveys of its construction can be found in \cite{johnson:survey} and
\cite{hain:normal}. To generate cubic relations in $\t_3$, we determine the
$\Sp(H)$-module of cubic relations determined by a certain pair of commuting
elements of $T_3$. Bounding pair maps give non-trivial elements of
$\Gr^W_{-1}\t_g$ via the Johnson homomorphism. In genus 3, one cannot find
distinct commuting bounding pair maps.\footnote{This provides a heuristic
explanation of why there are no quadratic relations in $\t_3$. In genus 4 and
beyond, commuting BP maps generate all relations in $\Gr^W_\dot\t_g$. Cf.\
\cite[\S10]{hain:torelli}.} Because of this, we are forced to consider a
commuting pair of elements of $T_3$ consisting of a bounding pair map and a Dehn
twist on a bounding simple closed curve (a BSCC map). For possible future
applications, we give a lower bound on the cubic relations generated by such
commuting pairs for all $g\ge 3$ as the computations are no more difficult than
they are in the genus 3 case. In the next section we will show that these are
all the cubic relations, which will allow us to give a presentation of $\t_3$.

Let $S$ be a reference surface of genus $g$. Throughout this section, $t_A$ will
be the BSCC map and $t_B:= t_{B_+}t_{B_-}^{-1}$ will be the BP map depicted in
Figure~\ref{fig:commuting_elets}.
\begin{figure}[ht!]
\epsfig{file=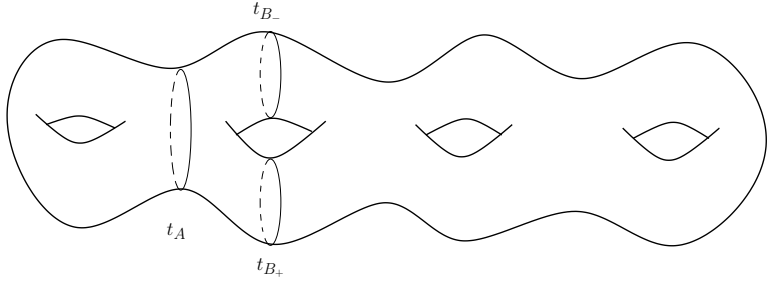, width=4in}
\caption{The BSCC map $t_A$ and the BP map $t_B:= t_{B_+}t_{B_-}^{-1}$}
\label{fig:commuting_elets}
\end{figure}
They commute as they have disjoint supports.

If we regard $t_A$ and $t_B$ as
elements of the unipotent completion $\T_g$ of $T_g$, we can take their
logarithms $\log t_A$ and $\log t_B$. These are commuting elements of $\t_g$.
Since $t_A$ lies in the kernel of the Johnson homomorphism and since
$\t_g=W_{-1}\t_g$,
$$
\log t_A \in W_{-2}\t_3 \text{ and } \log t_B \in W_{-1}\t_3.
$$
The image of their tensor product in $\Gr^W_{-2}\t_3 \otimes \Gr^W_{-1}\t_3$
thus lies in the kernel of the bracket
\begin{equation}
\label{eqn:bracket}
\Gr^W_{-2}\t_3 \otimes \Gr^W_{-1}\t_3 \to \Gr^W_{-3} \t_3.
\end{equation}
The ultimate goal of this section is to prove the following result.

\begin{proposition}
\label{prop:sub-mod}
If $g\ge 3$, then the $\Sp(H)$-submodule of the kernel of (\ref{eqn:bracket})
generated by the image of $\log t_A \otimes \log t_B$ contains the unique
submodule of $\Gr^W_{-2}\t_3 \otimes \Gr^W_{-1}\t_3$ with highest weights
$2\lambda_2+\lambda_3$, $\lambda_1+\lambda_2$, and $\lambda_3$.
\end{proposition}

Let $S'$ be the genus 1 subsurface of $S$ that lies to the left of the
separating SCC $A$. Choose a symplectic basis $a_1,\dots,a_g,b_1,\dots,b_g$ of
$H_1(S)$, where $a_2,b_2$ is a basis of $H_1(S')$ and where the homology classes
of $B_\pm$ are $\pm a_1$. The Johnson homomorphism induces an isomorphism
$$
\Gr^W_{-1}\t_g \cong \Lambda^3_0 H.
$$
The image of $\log t_B$ in $\Lambda^3_0 H$ under this isomorphism is the image
\begin{equation}
\label{eqn:johnson}
\overline{a_2\wedge b_2\wedge a_1} \in \Lambda^3_0 H \cong \Gr^W_{-1}\OutDer\p
\end{equation}
of $a_2\wedge b_2\wedge a_1 \in \Lambda^3 H$ in $\Lambda^3_0 H$. (Cf.\
\cite[Prop.~10.3]{hain:torelli}.)

Recall from Corollary~\ref{cor:outder} that $\Gr^W_{-2}\OutDer\p$ is isomorphic
to $V(2\lambda_2)$. Our first task is to compute the element of $V(2\lambda_2)$
that corresponds to $\log t_A$. The first step is to compute the image of a BSCC
map under the second Johnson homomorphism that takes an element of the kernel of
the first Johnson homomorphism to $\Gr^W_{-2}\OutDer \p$.

\subsection{The second Johnson homomorphism}

In this section we give an exposition of Morita's computation
\cite[p.~308]{morita:casson} of the second Johnson homomorphism. Suppose that
$g\ge 2$, that $x$ is a point in the reference surface $S$ and that $A$ is a
bounding simple closed curve in $S-\{x\}$. Denote the Lie algebra of the
unipotent completion of $\pi_1(S,x)$ by $\p$. Denote the Dehn twist about $A$ by
$t_A$. It is an element of $T_{g,1}$ and has logarithm $\log t_A$ in
$W_{-2}\t_{g,1}$. In this section we compute the image of $\log t_A$ in
$\Gr^W_{-2}\Der \p$.

\begin{figure}[ht!]
\epsfig{file=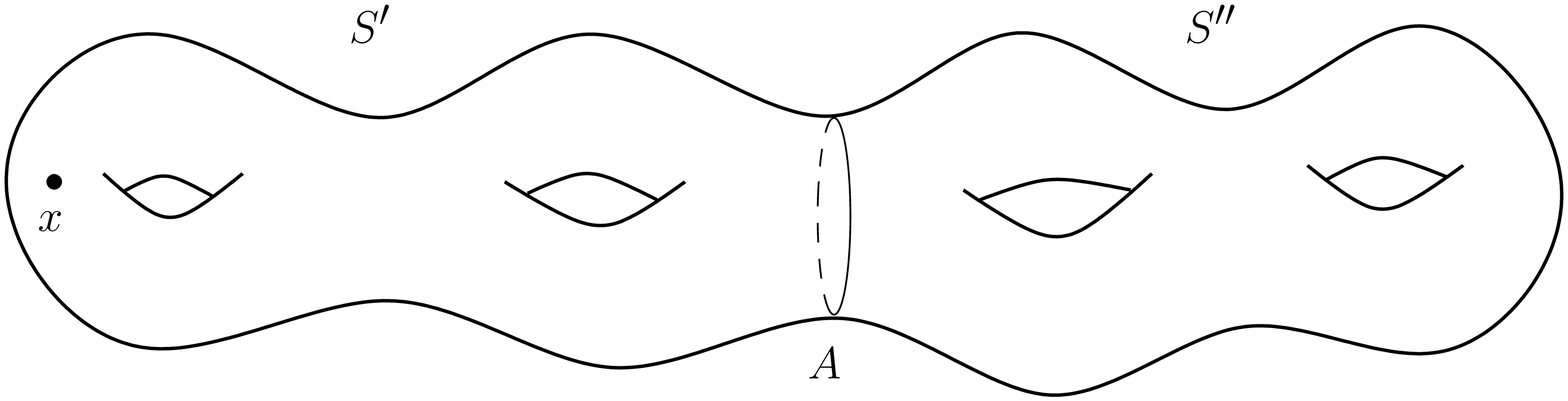, width=4in}
\caption{The BSCC $A$ and the decomposition $S=S'\cup S''$}
\end{figure}

The curve $A$ divides $S$ into two sub-surfaces. Denote the one containing $x$
by $S'$ and the other by $S''$. Let $g'$ and $g''$ be their respective genera.
Note that $g'+g'' = g$. Choose a symplectic basis
$a_1,\dots,a_{g'},b_1,\dots,b_{g'}$ of $H':= H_1(S')$ and
$a_{g'+1},\dots,a_{g},b_{g'+1},\dots,b_{g}$ of $H'':= H_1(S'')$. Set
$$
\theta' = \sum_{j=1}^{g'} a_j \wedge b_j,\
\theta'' = \sum_{j=g'+1}^{g} a_j \wedge b_j\text{ and }\theta =\theta'+\theta''.
$$
Then $H := H_1(S) = H'\oplus H''$ and $\theta$ corresponds to the cup product
$\Lambda^2 H^1(S) \to \Z$.

Define
$$
\phitilde : S^2 \Lambda^2 H \to \Hom(H,\L_3(H)) = \Gr^W_{-2}\Der\L(H)
$$
by
\begin{multline*}
\phitilde\big((u_1\wedge v_1)(u_2\wedge v_2)\big) :x \mapsto
\langle u_1,x \rangle[v_1,[u_2,v_2]] - \langle v_1,x \rangle[u_1,[u_2,v_2]]
\cr
+
\langle u_2,x \rangle[v_2,[u_1,v_1]] - \langle v_2,x \rangle[u_2,[u_1,v_1]].
\end{multline*}
This homomorphism is easily seen to be $\Sp(H)$-equivariant.

\begin{lemma}
For a subset $I$ of $\{1,\dots,g\}$, set $\theta_I = \sum_{i\in I}a_i\wedge
b_i$, $H_I = \Span\{a_i,b_i: i \in I\}$, and $H_I^c = \Span\{a_i,b_i: i \notin
I\}$. Then
$$
\phitilde(\theta_I^2)(x) =
\begin{cases}
0 & x\in H_I^c,\cr
2[x,\theta_I] & x \in H_I.
\end{cases}
$$
\end{lemma}

\begin{proof}
Set $\theta_j = a_j \wedge b_j$. The result follows immediately from the easily
verified fact that $\phitilde(\theta_j\theta_k)$ takes $x \in \Span\{a_i,b_i\}$
to $\d_{i,j}[x,\theta_k] + \d_{i,k}[x,\theta_j]$.
\end{proof}

\begin{corollary}
The homomorphism $\phitilde$ induces a well-defined $\Sp(H)$-equivariant
homomorphism
$$
\varphi : S^2\Lambda^2 H \to \Gr^W_{-2}\Der\p.
$$
\end{corollary}

\begin{proof}
The lemma implies that $\phitilde(\theta^2) = -2\ad_\theta$. This implies that
$\phitilde$ preserves the ideal $(\theta)$ in $\L(H)$ and thus descends to a
derivation of $\Gr^W_\dot \p = \L(H)/(\theta)$ of weight $-2$.
\end{proof}

\begin{proposition}[Morita {\cite[Prop.~1.1]{morita:casson}}]
The image of $\log t_A$ in $\Gr^W_{-2}\Der \p$ is $\varphi(\theta'')/2$.
\end{proposition}

\begin{proof}
Set $\pi = \pi_1(S,x)$.
First note that there are natural isomorphisms
$$
\Gr^W_{-2}\Der \p \cong \Hom(H,\L_3(H)/[H,\theta])
\cong \Hom(H_1(\pi),\Gr^3_\LCS\pi).
$$
Regard $\p$ as a subspace of the completion
$$
\Q\pi^\wedge := \varprojlim_m \Q\pi/J^m
$$
of the group algebra of $\pi$ with respect to the powers of its augmentation
ideal $J$, as explained in \cite[Appendix~A]{quillen}. The reason for working in
the completed group algebra is that it contains both $\pi$ and $\p$ and is acted
on by $\G_{g,1}$.

A standard and elementary fact is that the function $\pi \to J/J^2$ that takes
$\gamma$ to $(\gamma-1)+J^2$ is a homomorphism that induces an isomorphism $H =
H_1(\pi,\Q) \to J/J^2$. Since $T_{g,1}$ acts trivially on $H=J/J^2$, it follows
that $\sigma(\gamma)-\gamma \in J^2$ for all $\sigma \in T_{g,1}$ and
$\gamma\in\pi$. The induced homomorphism $T_{g,1} \to \Hom(J/J^2,J^2/J^3)$ is
essentially the Johnson homomorphism. (Cf.\ \cite[\S 11]{hain:normal}.) Since
$t_A$ lies in the kernel of the Johnson homomorphism,
$$
(t_A-1)(\gamma) = t_A(\gamma) - \gamma \in J^3
$$
for all $\gamma \in \pi$, so that $(t_A-1)^m : J^s\to J^{s+2m}$ for all $s,m>0$.
Consequently
\begin{align*}
\log(t_A)(\log(\gamma))
&=
\big(
(t_A-1) - (t_A -1)^2/2 + \cdots\big)\big((\gamma-1) - (\gamma-1)^2/2 + \cdots
\big)
\cr
&\equiv t_A(\gamma)-\gamma \bmod J^4
\end{align*}
We can thus compute the map
\begin{equation}
\label{eqn:logt}
\log t_A : J/J^2 \to J^3/J^4
\end{equation}
by computing $t_A(\gamma) - \gamma$ for a set of loops $\gamma$ that span $H$.

First, if $\gamma$ is the image of a loop in $(S',x)$, then
$t_A(\gamma)=\gamma$, which implies that the restriction of (\ref{eqn:logt}) to
$H'$ is trivial. To compute the restriction of (\ref{eqn:logt}) to $H''$, we
compute $t_A$ on loops $\gamma=\rho\mu\rho^{-1}$ that follow a fixed path $\rho$
in $S'$ from $x$ to a point $x'\in A$ and then follow a loop $\mu$ in $S''$
and return to $x$ along $\rho^{-1}$.

Denote by $\a$ the loop based at $x'$ that traverses $\partial S'$ in the
positive direction. Set $\b=\rho\a\rho^{-1}$. Then
$$
t_A(\gamma) = \rho\a^{-1}\mu\a \rho^{-1} =
\b^{-1}\gamma \b \equiv \{\b^{-1},\gamma\} + \gamma \bmod J^4,
$$
where $\{x,y\}$ denotes the group commutator $xyx^{-1}y^{-1}$.
Since
$$
\b^{-1} = \prod_{j=1}^{g'} \{\a_j,\b_j\},
$$
where $\a_j$, $\b_j$ is a standard set of generators of $\pi_1(S',x)$,
it follows
that
$$
\log t_A(\log \gamma) = t_A(\gamma) -\gamma
\equiv [\theta',\log\gamma]= -\ad_{\theta''}(\log\gamma).
$$
This implies that the restriction of $\log t_A$ to $H''$ is $-\ad_{\theta''}$.
\end{proof}

\subsection{A projection of $S^2\Lambda^2_0 H$ onto its $2\lambda_2$ isotypical
component}

The composite
$$
\xymatrix{
S^2\Lambda^2 H \ar[r]^(.45)\varphi & \Gr^W_{-2}\Der \p \ar[r] &
\Gr^W_{-2}\OutDer \p \cong V(2\lambda_2)
}
$$
is an $\Sp(H)$-equivariant map onto $V(2\lambda_2)$. The representation
$V(2\lambda_2)$ has multiplicity 1 in $S^2\Lambda^2 H$, and in $S^2\Lambda^2_0
H$. In this section, we compute a formula for the projection of $S^2\Lambda^2_0
H$ onto the $2\lambda_2$ isotypical factor.

Recall that $\Lambda^k_0 H$ denotes the $\Sp(H)$ module $\Lambda^k
H/(\theta\wedge \Lambda^{k-2}H)$. It is isomorphic to $V(\lambda_k)$. Note that
$\Lambda^2_0 H$ can be identified with the kernel of the polarization $\Lambda^2
H \to \Q$. The projection of $u\wedge v \in \Lambda^2 H$ onto this submodule is
$u\wedge v - \langle u,v \rangle\theta/g$.

\begin{lemma}
If $g\ge 2$ the cup product
$$
\mathrm{cup} : S^2 \Lambda^2_0 H \to \Lambda^4 H
$$
is a surjective $\Sp(H)$-module map. Its kernel is irreducible with highest
weight $2\lambda_2$. The $\Sp(H)$-module map
$$
\phi : \Lambda^4 H \to S^2 \Lambda^2_0 H
$$
that takes $x_1\wedge x_2 \wedge x_3\wedge x_4$ to
\begin{multline*}
\Big(x_1\wedge x_2 - \frac{\langle x_1,x_2\rangle}{g}\theta\Big)
\cdot \Big(x_3\wedge x_4 - \frac{\langle x_3,x_4\rangle}{g}\theta\Big)
\cr
- \Big(x_1\wedge x_3 - \frac{\langle x_1,x_3\rangle}{g}\theta\Big)
\cdot\Big(x_2\wedge x_4 - \frac{\langle x_2,x_4\rangle}{g}\theta\Big)
\cr
+ \Big(x_1\wedge x_4 - \frac{\langle x_1,x_4\rangle}{g}\theta\Big)
\cdot\Big(x_2\wedge x_3 - \frac{\langle x_2,x_3\rangle}{g}\theta\Big)
\end{multline*}
is injective. The restriction of $\mathrm{cup}\circ \phi :\Lambda^4 H \to
\Lambda^4 H$ to the component with highest weight $\lambda$ is multiplication by
$
\begin{cases}
3 & \lambda = \lambda_4,\ (g\ge 4\text{ only}),\cr
\frac{2g+2}{g} & \lambda = \lambda_2,\ (g\ge 3\text{ only}),\cr
\frac{2g+1}{g} & \lambda = 0.
\end{cases}
$
\end{lemma}

\begin{proof}
Kabanov's stability theorem \cite{kabanov:stability} implies that the highest
weight decomposition of $S^2 \Lambda^2_0 H$ is independent of $g$ when $g\ge
4$. One can show (using \LiE\ or by direct computation) that $S^2\Lambda^2_0 H$
is abstractly isomorphic to $V(2\lambda_2)\oplus \Lambda^4 H$ as an
$\Sp(H)$-module for all $g\ge 2$.

The cup product and $\phi$ are both $\Sp(H)$-equivariant. All assertions follow
from the last statement. Schur's Lemma implies that the restriction of
$\mathrm{cup}\circ \phi$ to each irreducible component of $\Lambda^4 H$ is
multiplication by a scalar. These scalars can be determined by computing the
composite on a highest weight vector of each irreducible component of $\Lambda^4
H$. These are $\theta^2$ ($g\ge 2$), $a_1\wedge a_2\wedge \theta$ ($g\ge 3$),
and $a_1\wedge a_2 \wedge a_3 \wedge a_4$ ($g\ge 4$). These are easily
computed. Details are left to the reader.
\end{proof}

\begin{corollary}
The projection of $(a_j\wedge b_j - \theta/g)^2 \in S^2 \Lambda^2_0 H$ onto the
weight $2\lambda_2$ isotypical component is
$$
(a_j\wedge b_j - \theta/g)^2
+ \frac{1}{g+1}\phi\Big(
a_j\wedge b_j \wedge \theta
-\frac{1}{(2g+1)}\theta^2
\Big)
$$
for all $g\ge 2$ and $1\le j \le g$.
\end{corollary}

\begin{proof}
The image of $(a_j\wedge b_j - \theta/g)^2$ in $\Lambda^4 H$ is
\begin{align*}
-2\,a_j\wedge b_j \wedge \theta/g + \theta^2/g^2
&=-\frac{2}{g}(a_j\wedge b_j \wedge \theta - \theta^2/g) - \theta^2/g^2\cr
&= -\mathrm{cup}\Big(\frac{1}{g+1}\phi(a_j\wedge b_j \wedge \theta - \theta^2/g)
+ \frac{1}{g(2g+1)}\phi(\theta^2)\Big) \cr
&= -\mathrm{cup}\circ\phi\Big(
\frac{1}{g+1}a_j\wedge b_j \wedge \theta
-\frac{1}{(g+1)(2g+1)}\theta^2
\Big)
\end{align*}
as the first term of the right-hand side lies in the $V(\lambda_2)$ component
and the second in the trivial component.
\end{proof}

Denote the projection of $x\in \Lambda^3 H$ (resp.\ $y \in S^2\Lambda^2_0 H$)
onto $\Lambda^3_0 H$ (resp.\ its $2\lambda_2$ isotypical component) by
$\overline{x}$ (resp.\ $\overline{y}$). Combining the previous result with
(\ref{eqn:johnson}), we obtain:

\begin{corollary}
For all $g\ge 3$, the images of $\log t_A \in W_{-1}\t_3$ and $\log t_B\in
W_{-2}\t_g$ in $\Gr^W_\dot\OutDer\p$ are
$$
\overline{\log t_A} = \overline{(a_2\wedge b_2-\theta/g)^2}
\in V(2\lambda_2) \cong \Gr^W_{-2}\OutDer\p
$$
and
$$
\overline{\log t_B} = \overline{a_2\wedge b_2\wedge a_1}
\in \Lambda^3_0 H \cong \Gr^W_{-1}\OutDer\p. \qed
$$
\end{corollary}

\subsection{Proof of Proposition~\ref{prop:sub-mod}}

As explained in the introduction to this section, we want to compute the
$\Sp(H)$-submodule of
$$
\Gr^W_{-2}\u_3 \otimes \Gr^W_{-1}\u_3 \cong V(2\lambda_2)\otimes V(\lambda_3)
$$
generated by the image of $\log t_A\otimes \log t_B$.

The following computation is equivalent to Sakasai's computation
{\cite[Lem.~6.1]{sakasai}}. It can be proved using $\LiE$\ plus Kabanov
stability.

\begin{lemma}
\label{lem:tensorprod}
The irreducible factors of $V(2\lambda_2)\otimes V(\lambda_3)$ have highest
weights
{\tiny$$
\begin{cases}
2\lambda_1+\lambda_3,\ \lambda_1+2\lambda_2,\
\lambda_1+\lambda_2,\ 2\lambda_2+\lambda_3,\ \lambda_3 & g=3, \cr
2\lambda_1+\lambda_3,\ \lambda_1+2\lambda_2,\ \lambda_1+\lambda_2+\lambda_4,\
\lambda_1+\lambda_2,\ \lambda_1 + \lambda_4,\ 2\lambda_2+\lambda_3,\ \lambda_2
+\lambda_3,\ \lambda_3 & g=4,\cr 
2\lambda_1+\lambda_3,\ \lambda_1+2\lambda_2,\ \lambda_1+\lambda_2+\lambda_4,\
\lambda_1+\lambda_2,\ \lambda_1 + \lambda_4,\ 2\lambda_2+\lambda_3,\ \lambda_2
+\lambda_3,\ \lambda_2+\lambda_5,\ \lambda_3 & g\ge 5.
\end{cases}
$$
}
Each has multiplicity 1. Consequently, there are unique (up to scalar
multiplication) projections of $V(2\lambda_2)\otimes V(\lambda_3)$ onto
$V(\lambda_3)$, $V(\lambda_1+\lambda_2)$ and $V(2\lambda_2+\lambda_3)$. \qed
\end{lemma}

Regard $V(2\lambda_2)$ as the kernel of the cup product $S^2\Lambda^2_0 H \to
\Lambda^4 H$. Regard $V(\lambda_3)=\Lambda^3_0 H$ as the kernel of the
contraction
$$
c : \Lambda^3 H \to H,\quad u\wedge v \wedge w \mapsto
\langle u,v \rangle w + \langle v,w \rangle u + \langle w,u \rangle v
$$
The projection of $x\in\Lambda^3 H$ onto $\Lambda^3_0 H$ is $x - c(x)\wedge
\theta/(g-1)$. One also has the injection (the ``Jacobi identity map'')
$$
\Lambda^3 H \to \Lambda^2 H \otimes H,\quad u\wedge v \wedge w \mapsto 
(u\wedge v)\otimes w + (v\wedge w)\otimes u + (w\wedge u) \otimes v
$$
and the standard imbedding $v_1v_1 \mapsto v_1\otimes v_2 + v_2 \otimes v_1$ of
$S^2\Lambda^2 H$ into $(\Lambda^2 H)^{\otimes 2}$. Finally, one has the pairing
$$
(\blank,\blank) : \Lambda^2 H\otimes \Lambda^2 S^2 H \to \Q.
$$
$$
(u_1\wedge v_1, u_2 \wedge v_2) =
\left|
\begin{matrix}
\langle u_1, v_1 \rangle & \langle u_1, v_2 \rangle \cr
\langle u_2, v_1 \rangle & \langle u_2, v_2 \rangle
\end{matrix}\right|
$$

All of these maps are $\Sp(H)$-equivariant. They can be assembled into the
$\Sp(H)$-equivariant map:
\begin{equation}
\label{eqn:projection}
V(2\lambda_2)\otimes V(\lambda_3)\hookrightarrow
S^2\Lambda^2 H \otimes \Lambda^3 H \hookrightarrow
\Lambda^2 H \otimes (\Lambda^2 H \otimes \Lambda^2 H) \otimes H
\overset{1\otimes\det\otimes 1}{\longrightarrow} \Lambda^2 H \otimes H.
\end{equation}
This can be composed with the multiplication map $\Lambda^2 H \otimes H \to
\Lambda^3 H$ to obtain a projection to $V(2\lambda_2)\otimes V(\lambda_3) \to
\Lambda^3 H$.

\begin{lemma}
\label{lem:span}
For all $g\ge 3$, the $\Sp_g$-submodule $W$ of $V(\lambda_2)\otimes
V(\lambda_3)$ generated by
$$
v := \overline{\log t_A} \otimes \overline{\log t_B} = 
\overline{(a_2\wedge b_2-\theta/g)^2} \otimes \overline{a_2\wedge b_2\wedge a_1}
$$
contains the factors with highest weights $2\lambda_2+\lambda_3$,
$\lambda_1+\lambda_3$ and $\lambda_3$. 
\end{lemma}

\begin{proof}
Set $\theta_j = a_j \wedge b_j$, so that $\theta = \sum_j \theta_j$. Note
that $(\theta_j,\theta_k) = \delta_{j,k}$. Since
$$
F_{2,3}\circ F_{1,2}^2(v) = 
2(a_1\wedge a_2)^2\otimes (a_1\wedge a_2 \wedge a_3) \in W
$$
is a highest weight vector of weight $2\lambda_2+\lambda_3$, $W$ contains the
copy of $V(2\lambda_2+\lambda_3)$.

To show that $W$ contains  highest weight $\lambda_1+\lambda_2$ we first apply
$F_{1,2}$ to $v$ to get
$$
\Big(2(a_2\wedge a_1)(\theta_2 - \theta/g)
+\frac{1}{g+1}\phi(a_2\wedge a_1\wedge \theta)\Big)
\otimes
\big(\theta_2 \wedge a_1 - \theta\wedge a_1/(g-1)\big) \in W.
$$
This goes to the vector
\begin{align*}
\bigg(
&2\Big(
(\theta_2,\theta_2) -\frac{1}{g}(\theta,\theta_2)
-\frac{1}{g-1}(\theta_2,\theta) + \frac{1}{g(g-1)}(\theta,\theta-\theta_1)
\Big)
\cr
&+ \frac{1}{g+1}\sum_{j>2}(\theta_j - \theta/g,\theta_2)
-\frac{1}{(g+1)(g-1)}\sum_{j>2}(\theta_j-\theta/g,\theta-\theta_1)
\Big)(a_2\wedge a_1)\otimes a_1
\cr
=& \bigg(2\Big(1-\frac{1}{g}-\frac{1}{g-1}+\frac{g-1}{g(g-1)}\Big)
-\frac{g-2}{g(g+1)}-\frac{g-2}{g(g-1)(g+1)}\bigg)(a_2\wedge a_1)\otimes a_1 \cr
=& -\frac{(g-2)(2g+1)}{g^2-1}\, (a_1\wedge a_2)\otimes a_1 \hspace*{202pt}
\end{align*}
in $\Lambda^2 H \otimes H$ under the projection (\ref{eqn:projection}). This is
a non-zero highest weight vector of weight $\lambda_1+\lambda_2$ for all $g\ge
3$.

Since we do not need to know the result for $\lambda_3$, we only sketch the
argument. Since $\theta\wedge a_1 \in \Lambda^3 H$ is a highest weight vector
with highest weight $\lambda_1$, and since $V(2\lambda_2)\otimes V(\lambda_1)$
does not contain $V(\lambda_3)$, it suffices to show that the $\Sp(H)$-module
generated by the image of
$$
w := \overline{(a_2\wedge b_2-\theta/g)^2} \otimes (a_2\wedge b_2\wedge a_1)
\in V(2\lambda_2)\otimes \Lambda^3 H
$$
in $\Lambda^3 H$ under the projection described above contains $V(\lambda_3)$.
To do this, first apply $F_{2,3}$ to $w$ to get
\begin{multline*}
\Big(2 (a_2\wedge a_3)(\theta_2-\theta/g) +
\frac{1}{g+1}\phi(a_2\wedge a_3 \wedge \theta)\Big)\otimes (\theta_2\wedge a_1)
\cr
+ \bigg((\theta_2-\theta/g)^2 + \frac{1}{g+1}\phi(\theta_2\wedge \theta)
- \frac{1}{(g+1)(2g+1)}\phi(\theta^2)\bigg)\otimes (a_1\wedge a_2 \wedge a_3).
\end{multline*}
This is a vector of weight $\lambda_3$ whose image in $\Lambda^3 H$ is the
highest weight vector
\begin{multline*}
\bigg(\frac{2(g-1)}{g} + \frac{g-1}{g(g+1)} + 0 + \frac{2}{g+1}
- \frac{6}{(g+1)(2g+1)}\bigg)\, a_1\wedge a_2 \wedge a_3 \cr
=  (g-1)\frac{4 g^2  + 12 g + 3}{g (g + 1) (2 g + 1)}\,a_1\wedge a_2 \wedge a_3.
\end{multline*}
Since this is non-zero for all $g\ge 3$, it follows that $W$ contains a factor
with highest weight $\lambda_3$.
\end{proof}

\section{Presentations of $\u_3$ and $\t_3$}
\label{sec:pres_3}

Recall that $\t_3$ is isomorphic to the degree completion of the graded Lie
algebra associated to its weight filtration. So a presentation of its associated
graded Lie algebra also gives a presentation of $\t_3$ itself. Similarly, a
presentation of $\Gr^W_\dot\u_3$ gives a presentation of $\u_3$.

\begin{theorem}
A minimal presentation of $\Gr^\dot_\LCS \t_3$ is
$$
\Gr^\dot_\LCS \t_3 = \L(V)/\r,
$$
where $V := V(\lambda_3)$ and $\r$ is the graded ideal generated by the unique
$\Sp_3$-submodule $R$ of $\L_3(V)$ isomorphic to
$$
V(2\lambda_2+\lambda_3)+V(\lambda_1+\lambda_2) + V(\lambda_3). \qed
$$  
\end{theorem}

The copy of $V(\lambda_3)$ in the cubic relations corresponds to the fact that
the generators $\Gr^W_{-1}\t_3$ commute with the copy of the trivial
representation $\Q\psi \in \Gr^W_{-2}\t_3$. That is, the $\lambda_3$ component
of the cubic relations is $[\psi,\Gr^W_{-1}\t_3]$. Since $\psi = 0$ in $\u_3$,
we obtain the following presentation of $\Gr^W_\dot\u_3$.

\begin{corollary}
A minimal presentation of $\Gr^\dot_\LCS \u_3$ is
$$
\Gr^\dot_\LCS \u_3 = \L(V)/\r,
$$
where $V := V(\lambda_3)$ and $\r$ is the graded ideal generated in degrees $2$
and $3$. The space of quadratic relations is the unique copy of the trivial
representation in $\L_2(V) \cong \Lambda^2 V$. The space of cubic relations is
the unique $\Sp_3$-submodule of $\L_3(V)$ isomorphic to
$V(2\lambda_2+\lambda_3)+V(\lambda_1+\lambda_2)$. \qed
\end{corollary}

Corollary~\ref{cor:sakasai} implies that all relations in $\Gr^W_\dot\t_3$ are
of weight $-3$ --- i.e., are cubic. The theorem is thus an immediate consequence
of the following computation.

\begin{proposition}
\label{prop:h2}
There are $\Sp(H)$-module isomorphisms
$$
\Gr^W_3 H^2(\t_3) \cong V(2\lambda_2+\lambda_3) + V(\lambda_1+\lambda_2)
+ V(\lambda_3)
$$
and $\Gr^W_{-3}\t_3 \cong \Gr^W_{-3}\u_3 \cong V(2\lambda_1+\lambda_3)$.
Consequently,
$$
\Gr^W_3 H^2(\u_3) \cong V(2\lambda_2+\lambda_3) + V(\lambda_1+\lambda_2).
$$
\end{proposition}

\begin{proof}
We use the notation of Proposition~\ref{prop:sakasai}. The kernel of $\coker J
\to \Gr^W_{-3}\t_3$ is $\Gr^W_{-3}H_2(\t_3)$. It contains the $\Sp(H)$-submodule
generated by $(\log t_A)\otimes (\log t_B)$. We know that the image of $\coker
J$ in $\Gr^W_{-3}\t_3$ contains $V(2\lambda_1+\lambda_3)$.
Proposition~\ref{prop:sub-mod} implies that $\Gr^W_{-3}H_2(\t_3)$ is generated
as an $\Sp(H)$-module by the image of $\log t_A\otimes \log t_B$ and that
$\Gr^W_{-3}\t_3$ is irreducible with highest weight $2\lambda_1+\lambda_3$.

The corresponding statements for $\u_3$ are immediate consequences.
\end{proof}

\subsection{Massey products}

The computations above can be stated in terms of Massey triple products.
Suppose that $X$ is a space (such as the classifying space of $T_3$). Fix  a
coefficient ring $R$ for cohomology. The {\em Massey triple product} is the map
\begin{multline*}
\ker\big\{H^p(X)\otimes H^q(X)\otimes H^r(X) \to
\big(H^{p+q}(X)\otimes H^r(X)\big)\oplus\big(H^p(X)\otimes H^{q+r}(X)\big)\big\}
\cr
\to
H^{p+q+r-1}(X)/\big(H^{p+q-1}(X)\wedge H^r(X) + H^p(X) \wedge H^{q+r-1}(X)\big)
\end{multline*}
defined by
$[w_1]\otimes [w_2]\otimes [w_3] \mapsto [w_{12}\wedge w_3 + w_1\wedge w_{23}]$
where $w_1,w_2,w_3$ are cocycles on $X$ of degrees $a,b,c$, respectively and
$w_{12}$ and $w_{23}$ are cochains satisfying $d w_{12} = w_1\wedge w_2$ and
$d w_{23} = \pm w_2\wedge w_3$.

In particular, if $p=q=r=1$ and the cup product $H^1(X)^{\otimes 2} \to H^2(X)$
vanishes, one has a well defined Massey triple product
$$
H^1(X)^{\otimes 3} \to H^2(X).
$$
When $X$ is the classifying space of $T_3$, the cup product
$H^1(T_3;\Q)^{\otimes 2} \to H^2(T_3;\Q)$ vanishes by
Corollary~\ref{cor:cup_prod}. The Massey triple product map is defined and
$\Sp_3(\Z)$-equivariant.

\begin{proposition}
The image of the Massey triple product map
$$
H^1(T_3;\Q)^{\otimes 3} \to H^2(T_3;\Q)
$$
is an $\Sp_3(\Z)$ module isomorphic to the restriction of
$V(2\lambda_2+\lambda_3) + V(\lambda_1+\lambda_2) + V(\lambda_3)$ to
$\Sp_3(\Z)$. Consequently, $\dim H_2(T_3,\Q) \ge 694$.
\end{proposition}

\begin{proof}
The corresponding statement for $\t_3$ follows by examining the weight 3 part of
the Chevalley-Eilenberg cochains on $\t_3$. It and the fact that $H^2(\t_3)$ is
pure of weight 3 imply that the Massey triple product map $H^1(\t_3)^{\otimes 3}
\to H^3(\t_3)$ is surjective. The assertion for $T_3$ follows as the map
$H^\dot(\t_3) \to H^\dot(T_3)$ is $\Sp_3$-equivariant, preserves Massey products
and is an isomorphism in degree 1 and injective in degree 2.
\end{proof}

\section{Genus $3$ Mapping Class Groups are not K\"ahler Groups}
\label{sec:kahler}

In this section we derive some general consequences of a mapping class group of
genus $g\ge 3$ being a K\"ahler group. In particular, we complete the proof
of Theorem~\ref{thm:main_mcg}.

To prove Theorem~\ref{thm:main_mcg} we need to know that if the genus 3 mapping
class group $\G$ is the fundamental group of a compact K\"ahler manifold $X$,
then the standard representation $\rho: \G \to \Sp(H)$ is the monodromy
representation of a PVHS. This will imply that the completion of $\G$ relative
to $\rho$ has a MHS and, by Theorem~\ref{thm:u_quad}, that the Lie algebra of
its prounipotent radical is quadratically presented, a contradiction. To prove
that $\rho$ is the monodromy representation of a PVHS, we appeal to Simpson's
fundamental work \cite{simpson}.

\subsection{Rigid representations of K\"ahler groups}

Recall that a representation $\rho : \G \to R$ of a finitely generated group
into an algebraic group $G$ over $\R$ is {\em rigid} if every representation
$\phi$ that is sufficiently close to $\rho$ in $\Hom(\G,R)$ is conjugate to
$\rho$. It is {\em properly rigid} if it is rigid as a homomorphism to the
Zariski closure of its image.

A well-known result of Weil \cite{weil} implies that $\rho$ is rigid if
$H^1(\G,\Ad(\rho)) = 0$, where $\Ad(\rho)$ is the representation
$$
\xymatrix{\G \ar[r]^\rho & G \ar[r]^(.4){\Ad} & \Aut(\g)}
$$
of $\G$ on the Lie algebra of $G$.

The following result is a special case of a result \cite[Thm.~5]{simpson} of
Simpson.

\begin{theorem}[Simpson]
\label{thm:simpson}
Suppose that $X$ is a connected compact K\"ahler manifold and that $\rho :
\pi_1(X,x) \to \GL(V)$ is an irreducible representation on a rational vector
space. If $H^1(X,\Hom(V,V))=0$, then 
\begin{enumerate}

\item $\rho$ is rigid and properly rigid,

\item the Zariski closure $R$ of $\rho$ is a reductive group of Hodge type,

\item $\rho : \pi_1(X,x) \to R$ is the monodromy representation of a  polarized
$\Q$-VHS over $X$ with fiber $V\oplus V^\ast$ over the base point $x$. 

\end{enumerate}
\end{theorem}

Combining this with the main results of \cite{hain:malcev}, we obtain:

\begin{corollary}
\label{cor:kahler}
With the same assumptions as in the preceding theorem, the completion of
$\pi_1(X,x)$ relative to $\rho$ has a natural $\Q$-MHS and the action of
$\pi_1(X,x)$ on $H_1(\u_X)$ is the monodromy representation of a polarized
$\Q$-VHS. \qed
\end{corollary}

\subsection{Rigidity for mapping class groups}

For the rest of this section, we suppose that $\G$ denotes $\G_{g,n}$ or
$\Ghat_{g,n}$ and that $\G'$ is a subgroup of $\G$ that contains its Johnson
subgroup. Corollary~\ref{cor:isom_comp} implies that the completion of $\G'$
relative to $\rho$ is isomorphic to the completion of $\G$ relative to $\rho$.
Denote it by $\cG$.

\begin{lemma}
If $g\ge 3$, then $\rho : \G' \to \GL(H)$ is rigid and properly rigid.
\end{lemma}

\begin{proof}
This follows from \cite{weil}. Since $\Ad(\rho) \cong H^{\otimes 2} \cong S^2 H
\oplus \Lambda^2_0 H \oplus \Q$, Proposition~\ref{prop:putman_mcg} implies that
both $H^1(\G',\Ad(\sp_g))$ and $H^1(\G',\Ad(\End(H)))$ vanish.
\end{proof}

Combining Corollary~\ref{cor:kahler} with Theorem~\ref{thm:u_quad} we obtain:

\begin{proposition}
If $\G'$ is the fundamental group $\pi_1(X,x)$ of a compact K\"ahler manifold
$X$, then the Lie algebra $\g_{X,x}$ of the completion $\cG_{X,x}$ of
$\pi_1(X,x)$ relative to $\rho$ has a natural MHS. The weight filtration of the
Lie algebra $\u_{X,x}$ of its prounipotent radical is its lower central series:
$$
W_{-m}\u_{X,x} = L^m \u_{X,x}.
$$
\end{proposition}

The fact that $\u_3$ is not quadratically presented completes the proof of
Theorem~\ref{thm:main_mcg} via Corollary~\ref{cor:quad_pres}.

\begin{corollary}
If $g=3$, then $\G'$ cannot be the fundamental group of a compact K\"ahler
manifold.
\end{corollary}

\section{Speculation}

The proof that genus $3$ mapping class groups are not K\"ahler clearly fails
when $g>3$ as $\u_g$ is quadratically presented by Theorem~\ref{thm:genus>3}. So
the question arises as to whether higher genus mapping class groups are
K\"ahler. My instincts tell me that they are not, but I have no serious
evidence to support this.

One way to attempt to prove that a mapping class group $\G_{g,n}$ (or a group
commensurable with it) is K\"ahler is to construct a smooth complete subvariety
$X$ of some smooth finite (orbifold) covering $\M$ of $\M_{g,n}$ such that the
inclusion $X \to \M$ induces an isomorphism on fundamental groups. One can
attempt to do this using the generalized Lefschetz Theorem \cite[p.~150]{smt} as
the Satake compactification of $\M$ is projective. This implies that the
inclusion $X \hookrightarrow \M$ of a generic 2-dimensional linear section $X$
of $\M$ induces an isomorphism on fundamental groups. However, such an $X$
cannot be compact as one of the boundary components of $\M$ has codimension 2.
The known complete subvarieties of $\M$ are far from generic and have special
fundamental groups.

Suppose that $X$ is a subvariety of an orbifold covering $\M$ of $\M_{g,n}$.
Denote the Zariski closure of the image of $\pi_1(X,x) \to \Sp(H)$ by $R$ and
the completion of $\pi_1(X,x)$ relative to $\pi_1(X,x) \to R$ by $\cG_{X,x}$.
Denote the completion of $\pi_1(\M,x) \to \Sp(H)$ by $\cG_{\M,x}$.

\begin{conjecture}
If $X$ is a smooth complete subvariety of a finite orbifold covering $\M \to
\M_{g,n}$ of $\M_{g,n}$, then $\pi_1(X,x) \to \pi_1(\M,x)$ is not an
isomorphism. Even stronger, the induced homomorphism $\cG_{X,x} \to \cG_{\M,x}$
cannot be an isomorphism.
\end{conjecture}

Suppose that $X$ is a smooth (quasi-) projective variety with fundamental group
a finite index subgroup $\G$ of $\G_{g,n}$. Let $\M$ be the finite orbifold
covering of $\M_{g,n}$ with fundamental group $\G$. On a more speculative level,
one can ask if there must be a morphism $X \to \M$ that induces an isomorphism
on fundamental groups. Here one may need to insist that $g$ be sufficiently
large.

One way to approach this problem is to ask whether there is a family of
principally polarized abelian varieties (PPAVs) $f: A\to X$ whose associated
local system $R^1f_\ast\Z$ is the local system $\H$ corresponding the
fundamental representation $\rho : \G \to \G_{g,n}\to \Sp_g(\Z)$. The PVHS
produced by Simpson's Theorem (Thm.~\ref{thm:simpson}) may not have Hodge level
one\footnote{The level of a Hodge structure $V=\oplus V^{p,q}$ is the maximum of
the $p-q$ for which $V^{p,q}\neq 0$.} and may not have rank $2g$, so that it
may not arise from a family of PPAVs over $X$. Note that if $X$ is projective,
then the local system whose fiber over $x\in X$ is $H_1(\u_{X,x})$ is a PVHS
over $X$ which is is isomorphic to the local system $\Lambda^3_0 \H$.

We now enter the realm of the wildly speculative. If such a family $A \to X$ of
PPAVs does exist, one can speculate that it is isogenous to a family of
jacobians. This is not totally unreasonable as the relations in $\u_g$ seem to
be related to controlling the image of the period mapping $\M_g \to \A_g$. One
might call this the ``topological Schottky problem''. If it has a positive
solution, we can assume that $A \to X$ is a family of jacobians of curves of
compact type. The final step would be to argue that one can lift the period map
$X \to \A_g$ to a morphism $X\to \M$.

\appendix

\section{Mapping Class Groups in Low Genus are not K\"ahler}
\label{sec:low_genus}

For completeness we prove that no finite index subgroup of a mapping class group
in genus $0$, $1$ or $2$ is K\"ahler. As remarked in the introduction, the genus
$2$ case is due to Veliche \cite{veliche} and the genus 0 and 1 cases appear to
be folklore. Since I could not find references for the genus 0 and 1 cases, I am
including complete proofs here. By adapting Veliche's argument in genus 2, we
will prove the more general statement that groups commensurable with
hyperelliptic mapping class groups are not K\"ahler.

The hyperelliptic mapping class groups $\D_{g,n}$ is the orbifold fundamental
group of the moduli stack $\cH_{g,n}$ of $n$-pointed hyperelliptic curves of
genus $g\ge 2$. A topological definition will be given below. Since every curve
of genus $2$ is hyperelliptic, $\D_{2,n}=\G_{2,n}$. So the genus 2 case is a
special case of the hyperelliptic case.

\begin{theorem}
\label{thm:low_genus}
Suppose that $g$ and $n$ are non-negative integers satisfying $2g-2+n>0$. If
$g\le 2$, then no finite index subgroup of the mapping class group $\G_{g,n}$
can be a K\"ahler group. If $g\ge 2$, then no finite index subgroup of the
hyperelliptic mapping class group $\D_{g,n}$ can be a K\"ahler group.
\end{theorem}

The key point in the proof is that hyperelliptic mapping class groups and
mapping class groups in genera $\le 1$ can be written as an extension of a
non-abelian free group by a finitely generated group. So, since non-abelian free
groups have infinitely many ends, the main technical ingredient of the proof is
the following result of Arapura, Bressler and Ramachandran.

\begin{theorem}[{Arapura, Bressler, Ramachandran \cite{abr}}]
\label{thm:not_kahler}
If $\G$ is an extension
$$
1 \to K \to \G \to G \to 1
$$
of a group with infinitely many ends by a finitely generated  group, then $\G$
is not a K\"ahler group.
\end{theorem}

Note that a finite index subgroup of such a group $\G$ is also an extension of a
group with infinitely many ends by a finitely generated group. So no finite
index subgroup of such a group $\G$ can be a K\"ahler group. Note also that the
class of such groups is closed under extension by a finitely generated group.

Suppose that $n\ge 1$ and that $U$ is a topological space. Denote the
configuration space of $n$ (labelled) points on $U$ by $C_n(U)$:
$$
C_n(U) = U^n - \text{(the fat diagonal)}
$$
where the fat diagonal consists of all $(u_1,\dots,u_n)\in U^n$ where the $u_j$
are not distinct.

\subsection{Genus $0$}
Note that $\M_{0,4}$ is isomorphic to $\C-\{0,1\}$; the point $t\in \C-\{0,1\}$
corresponds to the $4$-pointed rational curve $(\P^1;0,1,\infty,t)$. It follows
that $\G_{0,4}$ is free of rank 2 and therefore not K\"ahler. Suppose that $n\ge
4$. The fiber of the projection $\M_{0,n} \to \M_{0,4}$ that forgets all but the
first 4 points is a fibration. Its fiber over the point of $t\in \C-\{0,1\}$ is
the configuration space $C_{n-4}(\C-\{0,1\})$. This is a hyperplane complement
and has finitely generated fundamental group. So $\G_{0,n}$ is an extension of a
free group of rank 2 by a finitely generated group. Theorem~\ref{thm:not_kahler}
implies that no finite index subgroup of $\G_{0,n}$ can be K\"ahler.

\subsection{Genus $1$}
The genus 1 mapping class group $\G_{1,1}$ is isomorphic to $\SL_2(\Z)$. Since
$\SL_2(\Z)$ is virtually free, none of its finite index subgroups can be
K\"ahler. Let $G$ be a finite index free subgroup of $\G_{1,1}$. All such $G$
are free of rank $\ge 2$. Suppose that $n>1$. The inverse image $\G$ of
$G$ in $\G_{1,n}$ is an extension
$$
1 \to K \to \G \to G \to 1
$$
of $G$ by the fundamental group $K$ of the configuration space $C_{n-1}(E')$ of
a once punctured genus 1 surface. Since the group $K$ is finitely generated,
this implies that no finite index subgroup of $\G_{1,n}$ can be K\"ahler.

\subsection{Hyperelliptic mapping class groups}

We begin by recalling a few facts about hyperelliptic mapping class groups.
Suppose that $g\ge 2$ and that $S$ is a compact oriented surface of genus $g$. A
{\em hyperelliptic involution} $\sigma : S \to S$ is an orientation preserving
diffeomorphism of order 2 of $S$ with exactly $2g+2$ fixed points. By the
Riemann-Hurwicz formula, the quotient $S/\langle \sigma \rangle$ is a sphere.
From this it follows that all hyperelliptic involutions are conjugate in
$\Diff^+ S$.

The {\it hyperelliptic mapping class group} $\D_g$ is defined to be the set of
isotopy classes of orientation preserving diffeomorphisms of $S$ that commute
with $\sigma$:
$$
\D_g := \pi_0 (\text{centralizer of $\sigma$ in $\Diff^+ S$}).
$$
A result of Birman and Hilden \cite{birman-hilden} implies that the natural
homomorphism $\D_g \to \G_g$ is injective and that its image is the centralizer
of the isotopy class of $\sigma$ in $\G_g$.

Denote the set of fixed points of $\sigma$ by $W$. The natural homomorphism
$\rho_W : \D_g \to \Aut W$ is surjective. There is an obvious homomorphism
$\ker \rho_W \to \G_{0,2g+2}$ whose kernel is the subgroup generated by the
hyperelliptic involution. One therefore has an extension:
$$
1 \to \langle \sigma \rangle \to \ker \rho_W \to \G_{0,2g+2} \to 1.
$$
Since $\G_{0,2g+2}$ is an extension of a non-abelian free group by a finitely
generated group, so is $\ker\rho_W$. Theorem~\ref{thm:not_kahler} implies that
no finite index subgroup of $\ker\rho_W$ (and thus of $\D_g$) can be K\"ahler.

The hyperelliptic mapping class group $\D_{g,n}$ is defined to be the
restriction of the extension
$$
1 \to \pi_1(C_n(S),\ast) \to \G_{g,n} \to \G_g \to 1
$$
to $\D_g$. It is an extension
$$
1 \to \pi_1(C_n(S),\ast) \to \D_{g,n} \to \D_g \to 1.
$$
Since $\pi_1(C_n(S),\ast)$ is finitely generated, the inverse image of
$\ker\rho_W$ in $\D_{g,n}$ is an extension of a non-abelian free group by a
finitely generated group. Theorem~\ref{thm:not_kahler} implies that no finite
index subgroup of  $\D_{g,n}$ can be K\"ahler.


\begin{thebibliography}{99}

%\bibitem{amoros}
%J.~Amor\'os, M.~Burger, K.~Corlette, D.~Kotschick,
%D.~Toledo, D:
%{\em Fundamental groups of compact Kähler manifolds.} Mathematical Surveys and
%Monographs, 44. Amer.\ Math.\ Soc., 1996.

\bibitem{abr}
D.~Arapura, P.~Bressler, M.~Ramachandran:
{\em On the fundamental group of a compact K\"ahler manifold}, Duke Math.\ J.\
68 (1992), 477--488. 

\bibitem{asada-kaneko}
M.~Asada, M.~Kaneko:
{\em On the automorphism group of some pro-$\ell$ fundamental groups},
in Galois Representations and Arithmetic Algebraic Geometry, editor Y.~Ihara,
Advanced Studies in Pure Mathematics 12 (1987), 137-159.

\bibitem{birman-hilden}
J.~Birman, H.~Hilden:
{\it On the mapping class groups of closed surfaces as covering spaces}, in
``Advances in the theory of Riemann surfaces,'' 81--115, Ann.\ of Math.\
Studies, No.~66. Princeton Univ.\ Press, 1971.

\bibitem{boggi}
M.~Boggi:
{\em Fundamental groups of moduli stacks of stable curves of compact type},
Geom.\ Topol.\ 13 (2009), 247--276.

\bibitem{boggi-pikaart}
M.~Boggi, M.~Pikaart:
{\em Galois covers of moduli of curves}, Compositio Math.\ 120 (2000)\,
171--191.

%\bibitem{bridson-haefliger}
%M.~Bridson, A.~Haefliger:
%{\em Metric spaces of non-positive curvature}, Grundlehren der Mathematischen
%Wissenschaften 319, Springer-Verlag, 1999. 

\bibitem{deligne:hodge2}
P.~Deligne:
{\em Th\'eorie de Hodge, II}, Inst.\ Hautes \'Etudes Sci.\ Publ.\ Math.\ No.~40
(1971), 5--57.

\bibitem{dgms}
P.~Deligne, P.~Griffiths, J.~Morgan, D.~Sullivan:
{\em Real homotopy theory of Kähler manifolds}, Invent.\ Math.\ 29 (1975),
245--274.

\bibitem{dhp}
A.~Dimca, R.~Hain, S.~Papadima:
{\em The abelianization of the Johnson kernel}, \comment{arXiv:1101.1392}

\bibitem{smt}
M.~Goresky, R.~MacPherson:
{\em Stratified Morse Theory}, Ergebnisse der Mathematik und ihrer Grenzgebiete
14, Springer-Verlag, 1988.

\bibitem{hain:hodge}
R.~Hain:
{\em The de Rham homotopy theory of complex algebraic varieties, I}, K-Theory 1
(1987), 271--324.

\bibitem{hain:comp}
R.~Hain:
{\em Completions of mapping class groups and the cycle $C-C^-$}, in Mapping
Class Groups and Moduli Spaces of Riemann Surfaces, C.-F.~B¨odigheimer and
R.~Hain, editors, Contemp.\ Math.\ 150 (1993), 75--105.

\bibitem{hain:normal}
R.~Hain:
{\em Torelli groups and geometry of moduli spaces of curves}, Current topics in
complex algebraic geometry (Berkeley, CA, 1992/93), 97--143, Math.\ Sci.\ Res.\
Inst.\ Publ., 28, Cambridge Univ. Press, 1995. 

\bibitem{hain:malcev}
R.~Hain:
{\em The Hodge de Rham theory of relative Malcev completion}, Ann.\ Sci.\
\'Ecole Norm.\ Sup.\ (4) 31 (1998), 47--92.

\bibitem{hain:torelli}
R.~Hain:
{\em Infinitesimal presentations of the Torelli groups}, J.\ Amer.\ Math.\ Soc.\
10 (1997), 597--651. 

\bibitem{hain:jams}
R.~Hain:
{\em Rational points of universal curves}, J.\ Amer.\ Math.\ Soc.\ 24 (2011),
709--769.

\bibitem{hain-looijenga}
R.~Hain, E.~Looijenga:
{\em Mapping class groups and moduli spaces of curves}, Algebraic
geometry---Santa Cruz 1995, 97--142,  Proc.\ Sympos.\ Pure Math., 62, Part 2,
Amer.\ Math.\ Soc., 1997.

\bibitem{johnson:survey}
D.~Johnson:
{\em An abelian quotient of the mapping class group ${\mathcal I}_g$}, Math.\
Ann.\ 249 (1980), 225--242.

\bibitem{johnson:fg}
D.~Johnson:
{\em The structure of the Torelli group, I: A finite set of generators for
$\I$}, Ann.\ of Math.\ (2) 118 (1983), 423--442.

\bibitem{johnson:h1}
D.~Johnson:
{\em The structure of the Torelli group, III. The abelianization of $T_3$},
Topology 24 (1985), 127--144.

\bibitem{kabanov:stability}
A.~Kabanov:
{\em Stability of Schur functors}, J.\ Algebra 195 (1997), 233--240. 

\bibitem{kabanov:purity}
A.~Kabanov:
{\em The second cohomology with symplectic coefficients of the moduli space of
smooth projective curves}, Compositio Math.\ 110 (1998), 163--186.

\bibitem{labute}
J.~Labute:
{\em On the descending central series of groups with a single defining
relation}, J.\ Algebra 14 1970, 16--23.

\bibitem{looijenga}
E.~Looijenga:
{\em Smooth Deligne-Mumford compactifications by means of Prym level
structures}, J.\ Algebraic Geom.\ 3 (1994),283--293.

\bibitem{morgan}
J.~Morgan:
{\em The algebraic topology of smooth algebraic varieties}, Inst.\ Hautes
\'Etudes Sci.\ Publ.\ Math.\ No. 48 (1978), 137--204.

\bibitem{morita:casson}
S.~Morita:
{\em Casson's invariant for homology 3-spheres and characteristic classes of 
surface bundles}, Topology 28 (1989), 305--323.

%\bibitem{morita:prospect}
%S.~Morita:
%{\em Structure of the mapping class groups of surfaces: a survey and a
%prospect}, Proceedings of the Kirbyfest (Berkeley, CA, 1998), 349--406,  Geom.\
%Topol.\ Monogr., 2, Geom.\ Topol.\ Publ., 1999.

\bibitem{putman}
A.~Putman:
{\em The Johnson homomorphism and its kernel},
\comment{http://arxiv.org/abs/0904.0467}

\bibitem{quillen}
D.~Quillen:
{\em Rational homotopy theory}, Ann.\ of Math.\ (2) 90 (1969), 205--295. 

\bibitem{raghunathan}
M.~Raghunathan:
{\em Cohomology of arithmetic subgroups of algebraic groups, I},
Ann.\ of Math.\ (2) 86 (1967), 409--424.

\bibitem{sakasai}
T.~Sakasai:
{\em The Johnson homomorphism and the third rational cohomology of the Torelli
group}, Topology Appl.\ 148 (2005), 83--111.

\bibitem{shafarevich}
I.~Shafarevich:
{\em Basic algebraic geometry}, Grundlehren der mathematischen Wissenschaften,
Vol.~213, Springer-Verlag, 1974.

\bibitem{simpson}
C.~Simpson:
{\em Higgs bundles and local systems}, Inst.\ Hautes \'Etudes Sci.\ Publ.\
Math.\ No. 75 (1992), 5--95.

\bibitem{sipe}
P.~Sipe:
{\em Some finite quotients of the mapping class group of a surface}, Proc.\
Amer.\ Math.\ Soc. 97 (1986), 512--524.

\bibitem{veliche}
R.~Veliche:
{\em Genus $2$ mapping class groups are not K\"ahler}, Proc.\ Amer.\ Math.\
Soc.\ 135, 1441--1447.

\bibitem{weil}
A.~Weil:
{\em Remarks on the cohomology of groups}, Ann.\ of Math.\ (2) 80 (1964),
149--157.

\bibitem{zucker}
S.~Zucker:
{\em Hodge theory with degenerating coefficients: $L_2$ cohomology in the
Poincar\'e metric}, Ann.\ of Math.\ (2) 109 (1979), 415--476.

\end{thebibliography}
\end{document}